\documentclass[reqno, 12pt]{amsart}
\usepackage{hyperref}
\usepackage{amsmath, amsthm, amssymb, mathtools, graphicx, paralist, cite, color}

\usepackage{comment}
 \usepackage{amsthm}
\usepackage[utf8]{inputenc}
\usepackage{mathtools}
\usepackage{amssymb}
\usepackage{amsmath} 
\usepackage{amsthm}
\theoremstyle{plain}
\newtheorem{theorem}{Theorem}
\newtheorem{proposition}{Proposition}
\newtheorem{lemma}{Lemma}

\newtheorem{conjecture}{Conjecture}

\theoremstyle{remark}
\newtheorem{remark}{Remark}
\usepackage{indentfirst}
\usepackage{appendix}
\usepackage{listings}
\usepackage{hyphenat}
\usepackage{calrsfs}
\usepackage{enumerate}
\usepackage{tikz-cd}
\usepackage{mathrsfs}
\usepackage{bbm}

\def \b {{\beta}}

\def \H {{\mathbb H}}

\def \psl  {{\hbox{PSL}_2( {\mathbb R})} }

\def \pslz  {{\hbox{PSL}_2( {\mathbb Z})} }
\def \di {{d^{(0)}_{t}(f)}}
\def \dii {{d^{(1)}_{t}(f)}}

\title{Refined Counting of Geodesic Segments in the Hyperbolic Plane}
\author{Marios Voskou}
\address{Alfr\'{e}d R\'{e}nyi Institute of Mathematics, POB 127, Budapest H-1364, Hungary}
\email{voskou.marios@renyi.hu}
\date{\today}


\begin{document}
\begin{abstract} For $\Gamma$ a cofinite Fuchsian group, and $l$ a fixed closed geodesic, we study the asymptotics of the number of those images of $l$ that have a prescribed orientation and distance from $l$ less than or equal to $X$. Using a new relative trace formula that we develop, we give a new concrete proof of the error bound $O(X^{2/3})$ that appears in the works of Good and Hejhal. Furthermore, we prove a  new bound $O(X^{1/2}\log{X})$ for the mean square of the error. For particular arithmetic groups, we provide interpretations in terms of correlation sums of the number of ideals of norm at most $X$ in associated number fields, generalizing previous examples due to Hejhal.
\end{abstract}
\maketitle
\section{Introduction}

For a fixed cofinite Fuchsian group $\Gamma \subset \psl$ and points $z_1,z_2 \in \mathbb{H}$, where $\mathbb{H}$ is the hyperbolic upper half-plane, the associated hyperbolic lattice counting problem is concerned with estimating the number \begin{equation} \label{clcountingf} N(X,z_1,z_2):= \# \left\{ \gamma \in \Gamma \vert  2\cosh{\rho(\gamma z_1, z_2) } \leq X \right\}.\end{equation}
Here, $\rho$ is the hyperbolic distance function, induced from the Poincar\'e metric $ds(z)^2=y^{-2}\left( dx^2+dy^2 \right)$, where $z=x+iy$. 

 Using the spectral theory of automorphic kernels, we can derive the following result \cite[Thm 12.1]{iwaniec}.
\begin{theorem}[Selberg \cite{selberg}, G\"{u}nther \cite{gunther}, Good \cite{good}]  \label{classiccounting}
Let $\Gamma$ be a cofinite Fuchsian group, and $z_1,z_2 \in \H$. Then,
$$N\left(X,z_1,z_2\right)= \sum_{1/2<s_j \leq 1} \pi^{1/2}\frac{\Gamma \left(s_j-\frac{1}{2} \right)}{\Gamma \left(s_j+1 \right)} u_{0,j}(z_1) \overline{u_{0,j}(z_2)}X^{s_j}+ O\left(X^{2/3} \right).$$
\end{theorem}
Here, the sum is over the small eigenvalues $\lambda_j=s_j(1-s_j)<1/4$ of the Laplacian of the hyperbolic surface $\Gamma \backslash \mathbb{H}$. Furthermore, $(u_{0,j})_j$ is a corresponding maximal $L^2$-orthonormal set of eigenfunctions.

 Let now $\Gamma_1=\left\langle \gamma_1\right\rangle$, $\Gamma_2=\left\langle \gamma_2\right\rangle$ be one\hyp{}generator subgroups of $\Gamma$. Considering the nine possible combinations of types for $\gamma_1$, $\gamma_2$ (parabolic, elliptic, or hyperbolic) we can associate the  space of double cosets  $\Gamma_1 \backslash \Gamma \slash \Gamma_2$ with nine distinct counting problems. These problems arise by considering the $\Gamma_i$'s as stabilizer groups of geodesics, cusps, or points for the hyperbolic, parabolic, and elliptic case respectively. For example, the elliptic-elliptic case recovers the hyperbolic lattice counting problem discussed above. In \cite{good}, Good attempts to solve all nine problems simultaneously. He provides asymptotic formulae for the corresponding counting functions, with a claimed asymptotic error term of order $O\left( X^{2/3} \right)$.
His notation and techniques are very complicated and therefore hard to verify. This makes it difficult to incorporate them in related problems, such as the refined study of the behaviour of the error term. Therefore, it is important to consider more concrete approaches for each problem separately.  

 Huber \cite{huber} investigates the hyperbolic-elliptic case, managing to prove an error term of order $O(X^{3/4})$. With a more careful examination of the transforms involved, Chatzakos--Petridis \cite{chatzakos} recovered the error term $O(X^{2/3})$ which appeared in Good \cite{good}.

We focus on the hyperbolic\hyp{}hyperbolic problem, corresponding to counting distances between a fixed geodesic $l_1$ (with stabilizer $\Gamma_1$) and the elements of the orbit of another fixed geodesic $l_2$ (with stabilizer $\Gamma_2$). For simplicity, we take $$\gamma_1=\gamma_2= \begin{pmatrix}
\lambda & 0\\
0 & \lambda^{-1}
\end{pmatrix}.$$
The corresponding closed geodesic is hence the vertical segment $l$ connecting $i$ with $\lambda^2\cdot i$.
The corresponding counting problem concerns estimating
$$N(X,l):=\#\left\{ \gamma \in \Gamma_1 \backslash \Gamma \slash \Gamma_2 \vert  \cosh{\hbox{dist}(\gamma l, l) } \leq X \right\}.$$
In \cite[Lemma 1]{mckee} (see also \cite[Lemma 1.1]{lekkas}), Martin--Mckee--Wambach show that,
 for $\gamma= \begin{pmatrix}
a & b\\
c & d
\end{pmatrix},$
we have $$\cosh{\hbox{dist}(\gamma l, l) }= \mathrm{max}\left(\left| B(\gamma) \right|,1\right), \,
\hbox{where } \; \; B(\gamma):=ad+bc.$$
Therefore, for $X>1$, we can write
\begin{equation} \label{lekkasn} N(X,l)=\#\left\{ \gamma \in \Gamma_1 \backslash \Gamma \slash \Gamma_2 \vert \left| B(\gamma) \right| \leq X \right\}. \end{equation}
Using methods similar to Huber, Lekkas \cite{lekkas} proved, independently of Good, the following result.
\begin{theorem}[Good \cite{good}, Lekkas \cite{lekkas}] \label{lekkasmain}  Let $\Gamma$ be a cocompact Fuchsian group with no elements having both diagonal entries equal to zero. For $N(X,l)$ as in equation (\ref{lekkasn}), we have
$$N(X,l)=\frac{2\left(\mathrm{len}(l)\right)^2}{\pi\mathrm{Vol}\left(\Gamma \backslash \mathbb{H} \right)}X+\sum_{1/2 < s_j < 1}D(s_j)\hat{u}_{0,j}^2X^{s_j}+O\left( X^{2/3} \right),$$ where $\mathrm{len}(l)=2 \log{\lambda}$ is the hyperbolic length of the geodesic segment $l$ and $$D(s):=\frac{\Gamma \left(s-1/2 \right)\Gamma \left(s/2+1/2 \right)}{\left(\Gamma \left(s/2 \right)\right)^2\Gamma \left(s/2+1 \right)}. $$
\end{theorem}
Here, the \emph{periods} $\hat{u}_{0,j}$ of the Maa{\ss} forms $u_{0,j}(z)$ are defined by
$$\hat{u}_{0,j}:=\int_{l}u_{0,j}(z)\, ds(z).$$
In addition, we define the periods $\hat{u}_{1,j}$ by
$$\hat{u}_{1,j}:=\int_{l}u_{1,j}(z)\, ds(z),$$
where, for $\lambda_j \neq 0$,
\begin{equation}\label{defofu1} u_{1,j}(z):=-2\lambda_j^{-1/2} \cdot \Im(z) \cdot \frac{\partial}{\partial z} u_{0,j}(z) \end{equation} are Maa{\ss} forms of weight $2$. For $\lambda_j=0$, we define $u_{1,j}(z)$ to be identically zero.

In \cite[p.116,~Thm.4]{good}, Good proves a more refined version of Theorem \ref{lekkasmain}. In particular, he provides separate estimates for the quantities 
$$N^{\mu,\mu'}\left(X \right):=\#\left\{ \gamma \in \Gamma_1 \backslash \Gamma \slash \Gamma_2 \vert  |ad+bc| \leq X, \; \mathrm{sign}(ab)=\mu, \; \mathrm{sign}(ac)=\mu' \right\},$$
where $\mu,\mu' \in \left\{ 1,-1\right\}$.

Geometrically, for $|ad+bc|>1$, the number $\mu$ corresponds to the direction of $\gamma l$ (clockwise for $\mu=1$ and anti-clockwise for $\mu=-1$), and the number $\mu'$ corresponds to which side of the imaginary axis $\gamma l$ lies in (positive for $\mu'=1$, negative for $\mu'=-1$) - see Figure \ref{figurefourc}. The case $|ad+bc| \leq 1$ corresponds to $\gamma l$ intersecting $l$. See for example \cite[Lemma 1]{mckee}.
\begin{figure}[ht]\label{figurefourc} \centering
\includegraphics[scale=0.5]{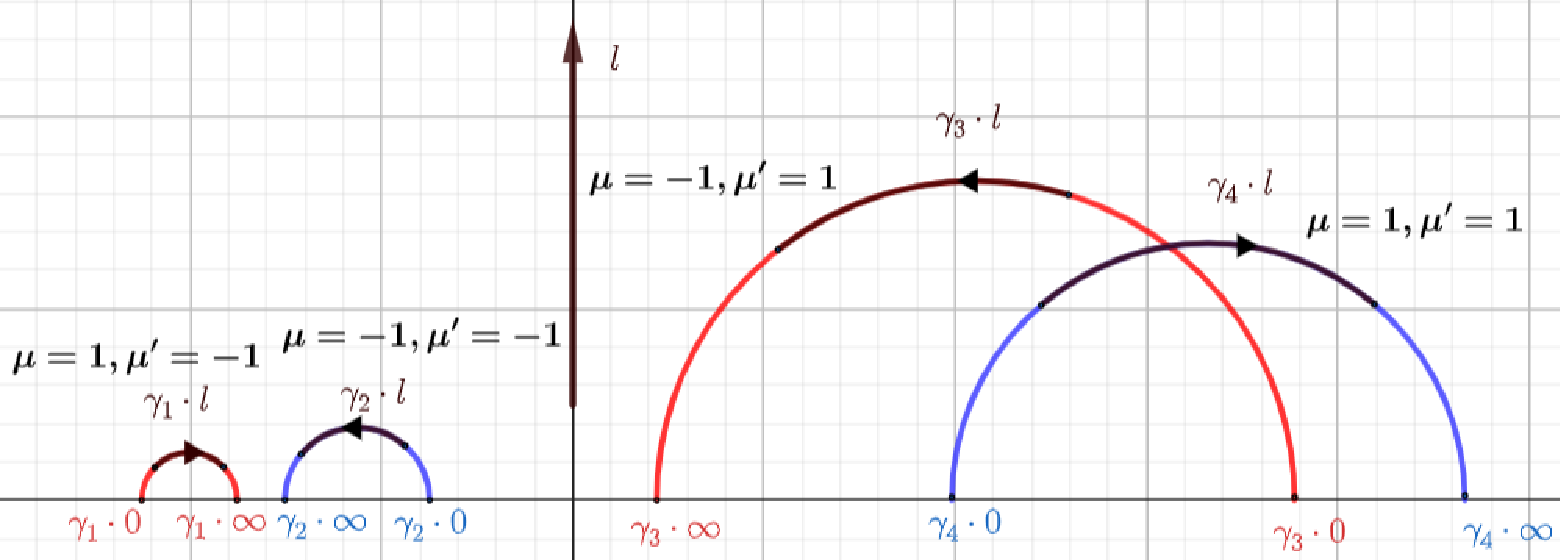}
\caption{Geometric interpretation of the four cases that correspond to the four possible choices of $(\mu,\mu')$. }
\end{figure}

In the series of papers \cite{hejhal},\cite{hejhal1},\cite{hejhal2},\cite{hejhal3} Hejhal states the asymptotic formulae for $N^{\mu,\mu'}$ and summarizes a strategy for the proof.

In the following sections we will establish such asymptotics, by first developing a new relative trace formula (see Theorem \ref{modtrace}). In particular, we prove the following theorem.
\begin{theorem}\label{mainthm}
Let $\Gamma$ be a fixed cofinite Fuchsian group. For $\mu,\mu' \in \left\{ 1,-1\right\}$, we have
$$N^{\mu,\mu'}(X)=M^{\mu,\mu'}(X)+O\left(X^{2/3} \right),$$
where
\begin{align*}M^{\mu,\mu'}(X)=\frac{\left(\mathrm{len}(l)\right)^2}{2\pi\mathrm{Vol}\left(\Gamma \backslash\mathbb{H}\right)}X+\frac{1}{4}\sum_{1/2<s_j < 1}  D(s_j)\left( \hat{u}_{0,j}+\mu\cdot  a_j\hat{u}_{1,j}\right)\left(\hat{u}_{0,j}-\mu' \cdot a_j\hat{u}_{1,j}\right)X^{s_j}, \end{align*} the coefficients $D(s_j)$ is as in Theorem \ref{lekkasmain}, and
$$a_j=a(s_j):=\frac{\sqrt{\lambda_j}}{2} \left(\frac{\Gamma(s_j/2)}{\Gamma(s_j/2+1/2)}\right)^2.$$

\end{theorem}
\begin{conjecture} \label{halfconj}
 The exponent $2/3$ in Theorem \ref{mainthm} can be improved to $1/2+\epsilon$.
\end{conjecture}

We are also interested in mean square errors. For example, in \cite{chamizo2}, Chamizo proves that the mean square error for the elliptic-elliptic counting problem is of order $O(X^{1/2}\log{X})$, as an application of the large sieve inequality he develops in \cite{chamizo1}. In \cite{chatzakos}, Chatzakos--Petridis use the same large sieve inequality to prove the same bound for the mean square error for the elliptic-hyperbolic counting problem. In \cite{voskoulekkas}, using Theorem \ref{modtrace}, we develop a large sieve inequality with weights the periods $\hat{u}_{m,j}$ instead of the values $u_{0,j}(z)$ from Chamizo. Lekkas \cite{lekkas} uses the case $m=0$ to prove a slightly worse upper bound for the mean square error term of the hyperbolic-hyperbolic problem. In this paper, we use these sieve inequalities to prove that the upper bound $X^{1/2}\log{X}$ is still valid for the hyperbolic-hyperbolic problem, even when its four cases are considered separately. This implies an averaged version of Conjecture \ref{halfconj}.
\begin{theorem}
   \label{newmserr}
     For $\mu,\mu' \in \left\{-1,1\right\}$, let \begin{equation} \label{errordef} E^{\mu,\mu'}(X):=N^{\mu,\mu'}(X)-M^{\mu,\mu'}(X),
     \end{equation}
     where
     $N^{\mu,\mu'}$ and $ M^{\mu,\mu'}$ are as in Theorem \ref{mainthm}.
     Then, as $X \rightarrow +\infty$, we have
    $$ \frac{1}{X}\int_{X}^{2X} \left|E^{\mu,\mu'}(x)\right|^2 \, dx\ll X\log^2{X}.$$
\end{theorem}
Finally, applying Theorems \ref{mainthm} and \ref{newmserr} for appropriate groups associated with quaternion algebras, we will prove the following theorem. For $p=5$, the first part is also stated by Hejhal (see \cite[Thm.2]{hejhal}).
\begin{theorem} \label{ideals} Let $p$ be fixed prime number, and let $c_p$ be a constant defined by
\begin{equation*}c_p :=p+\left(\frac{2}{p} \right) = \left\{
\begin{array}{ll}
      \displaystyle p-1, & p=  \pm 3\,  \left( \mathrm{mod} \; 8 \right), \\
      \displaystyle p+1, & p=  \pm 1 \,  \left( \mathrm{mod} \; 8 \right),\\
      \displaystyle p, & p=2. \\
\end{array} 
\right.
\end{equation*}We have
\begin{equation} \label{corras} \sum_{n \leq X}\mathcal{N}(n)\mathcal{N}(pn\pm 1)=\frac{4p}{c_p} \cdot    \left(\frac{\log{\epsilon}}{\pi}\right)^2X+\sum_{1/2 < s_j < 1}a^{\pm}_jX^{s_j}+E^{\pm}\left(X \right), \end{equation}
with
$$E^{\pm}\left(X\right)=O\left(X^{2/3}\right),$$
where $\mathcal{N}(n)$ is the number of ideals $\mathfrak{a}$ of $\mathbb{Z}\left[\sqrt{2}\right]$ with $N(\mathfrak{a})=n$, $\epsilon$ is the corresponding fundamental unit, and $a^{\pm}_{j}$ are real numbers.
Furthermore, we have
$$ \frac{1}{X}\int_{X}^{2X} \left|E^{\pm}(x)\right|^2 \, dx\ll X\log^2{X}.$$
\end{theorem}
\begin{remark} \label{jaclan}
For small values of $p$, in particular $p \in \left\{p: p<70\right\} \cup \left\{83, 101, 107, 109 \right\}$, the summation in the middle term in (\ref{corras}) is empty. This is due to proved cases of Selberg's $1/4$ eigenvalue conjecture. In particular, for the case $p = \pm 3 \pmod{8}$, from the explicit Jacquet--Langlands correspondence, the spectrum of the group $\Gamma$ defined in equation (\ref{quatgroup}) is the same to the spectrum of the group $\Gamma_0(8p)$ (see, for example, the proof of Hejhal in \cite{hejhalquat}).  On the other hand, for the cases $p = \pm 1 \pmod{8}$ and $p=2$, $\Gamma$ can be shown to be conjugate to a congruence group of level $4p$. In \cite{booker}, Booker, Lee, and Str{\"o}mbergsson verified Selberg's $1/4$ eigenvalue conjecture in $\Gamma_0(N)$ for $N<880$, and in $\Gamma(N)$ for $N<226$.
\end{remark}
\begin{remark}
With the same techniques, we can consider  $\mathbb{Q}\left(\sqrt{q}\right)$ with narrow class number $1$ instead of $\mathbb{Q}\left(\sqrt{2}\right)$.
\end{remark}
\subsection{Summary}
\,
In Section~\ref{setupsection}, we give a brief overview of the framework in which we will work. This includes relevant notation and definitions, as well as a geometric interpretation of the main idea behind the proof of Theorem~\ref{mainthm}.

In Section \ref{sesection}, we study the spectral expansion of the $A_f^{(1)}$ series defined in equation (\ref{hseries2}). In Section \ref{mrtfsection}, we use this spectral expansion to prove certain modified relative trace formulae (Theorem \ref{modtrace}), which we will use in the proofs of Theorems \ref{mainthm} and \ref{newmserr}. These are also a key ingredient in the proof the large sieve inequality given in Theorem \ref{sievevar}, which appears in \cite[Thm.3]{voskoulekkas}.

In Sections \ref{choicesection} to \ref{mainthmsection}, we prove an equivalent form of Theorem \ref{mainthm}, namely Theorem \ref{equivmain}. For the proof, our methods have many similarities with the ideas discussed in Hejhal \cite{hejhal,hejhal1,hejhal2,hejhal3}. 
In Section \ref{choicesection}, we choose kernels that make the geometric sides of Theorem \ref{modtrace} asymptotically equal to the quantities $N_i$ that appear in Theorem \ref{equivmain}. 
In Section \ref{estsection}, we provide estimates for the special functions appearing in the spectral side of our trace formulae, which we use in the proof of Theorem \ref{equivmain}.  In Section \ref{secperiods}, we prove upper bounds for the mean square of the periods $\hat{u}_{1,j}$, in a similar manner with the upper bound for the mean square of $\hat{u}_{0,j}$, see Huber \cite[Eq.63]{huber}. This is a weaker version of \cite[Thm 1]{hejhal2}. The stronger version is not necessary for our arguments.
In Section \ref{mainthmsection}, we finish the proof of Theorem \ref{equivmain} and, therefore, of Theorem \ref{mainthm}.

Then, in Section \ref{sectionmse}, we use the large sieve inequalities from \cite{voskoulekkas} and the estimates from Sections \ref{estsection} and \ref{secperiods} to prove Theorem \ref{newmserr} about the mean square of the error term.

Finally, in Section \ref{arithmeticsection}, we apply  Theorems \ref{mainthm} and \ref{newmserr} for certain arithmetic groups arising from quaternion algebras to deduce Theorem \ref{ideals}.
In Appendix \ref{spapp} we provide results for generalized hypergeometric functions that we use in Section \ref{estsection}.
\section{Preliminaries} \label{setupsection}

For a non-negative integer $m$, we denote by $\mathfrak{h}_m$ the space of $L^2$\hyp{}functions that transform as
$$F(\gamma z)=j^{2m}_{\gamma}(z)F(z)$$ under $\Gamma$, where, for $\gamma =\begin{pmatrix}
a & b \\
c & d 
\end{pmatrix}$ we have \begin{equation}j_{\gamma}(z):=\frac{cz+d}{\left|cz+d\right|}.  \label{jfactor} \end{equation}
Here, $L^2$ is the space of functions $f$ such that $\left \langle f, \, f \right \rangle$ is finite. The inner product  $\left \langle f, \, g \right \rangle$ is defined by
$$\left \langle f, \, g \right \rangle := \int_{\Gamma \backslash \mathbb{H}}f(z) \cdot \overline{g(z)} \; \,\frac{dx \, dy}{y^2}.$$

Let \begin{equation} \label{dms} D_m:=y^2 \left ( \frac{\partial ^2}{\partial x^2} + \frac{\partial^2}{\partial y^2} \right) - 2imy\frac{\partial}{\partial x} \end{equation} be the Laplacian in $\mathfrak{h}_m$.

Fix a maximal orthonormal set of real-valued $\mathfrak{h}_0$\hyp{}eigenfunctions $\left(u_{0,j}\right)_j$ for the discrete spectrum of $D_0$, with corresponding eigenvalues $\lambda_j=s_j(1-s_j)$. Let also $E_{\mathfrak{a}}\left(z, \, s\right)$ denote the Eisenstein series with respect to the cusp $\mathfrak{a}$ (see \cite[(3.11)]{iwaniec}). We define the Maa{\ss} raising operators by
$$K_m:=(z-\bar{z})\frac{\partial}{\partial z}+m.$$
It can be shown that $K_m$ maps $\mathfrak{h}_m$ to $\mathfrak{h}_{m+1}$ (see \cite[p.308]{roelcke}). Furthermore, for any $m$, the functions $\left(u_{m,j}\right)_j$ defined recursively by $$u_{m+1,j}:=\frac{i}{\sqrt{\lambda_j+m^2+m}} \cdot K_mu_{m,j}$$ form an orthonormal $\mathfrak{h}_m$\hyp{}eigenbasis for the discrete spectrum of $D_m$ with the same corresponding eigenvalues (see \cite[ p.146, eq.11]{fay}). This generalizes the definition of $u_{1,j}$ given in equation (\ref{defofu1}).
In a similar manner, for a cusp  $\mathfrak{a}$, we define
$$E_{\mathfrak{a},0}\left(z, \, s\right):=E_{\mathfrak{a}}\left(z, \, s\right), \quad E_{\mathfrak{a},m+1}\left(z, \, s\right):=\frac{i}{\sqrt{1/4+t^2+m^2+m}} \cdot K_m E_{\mathfrak{a},m}\left(z, \, s\right),$$
where $s=1/2+it$.
We further define the associated periods by
$$\hat{u}_{m,j}:=\int_{l}u_{m,j}(z) ds(z), \; \hat{E}_{\mathfrak{a}, m}(s'):=\int_{l}E_{\mathfrak{a},m}\left(z, \, s'\right) ds(z).$$

Let  $k:\H \times \H \longrightarrow \mathbb{C}$ be a sufficiently smooth and rapidly decaying function, where $k(z,w)$ is a function of $u(z,w)$, where
$2u(z,w)+1=\cosh{\left(\rho(z,w)\right)}.$ In a slight abuse of notation, we write $$k(u)=k(u(z,w))=k(z,w).$$
Define further the automorphization $$K(z,w):=\sum_{\gamma \in \Gamma}k(z,\gamma w),$$ called an \emph{automorphic kernel}.

In particular, for $k(u)=\mathbbm{1}_{[0,X]}(4u+2)$, we have
$K(z,w)=N(z,w,X)$ (see equation (\ref{clcountingf})). Theorem \ref{classiccounting} can be derived by considering the spectral expansion of the automorphic kernel corresponding to a smoothing of $k(u)$ (see \cite[Thm 12.1]{iwaniec}).

In our methods, we use repeatedly the system of coordinates $(u,v)$, called the \emph{Huber coordinates}, and defined by 
\begin{equation} \label{coordsys} u=\log{|z|}, \; \; \; \; v=-\arctan{\left(\frac{x}{y}\right)}, \nonumber \end{equation}
or, equivalently,
$$x=-e^u\sin{v}, \quad y=e^u\cos{v}. $$  Here,
$x$ and $y$ are the real and imaginary part of $z$ respectively.
We note that, for any diagonal matrix $\gamma= \begin{pmatrix}
\lambda & 0\\
0 & \lambda^{-1}
\end{pmatrix}$, we have $$v\left(\gamma z\right)=v(z),$$
and $$u\left(\gamma z \right)=u(z)+2\log{|\lambda|}.$$
Note further that $v(z)$ can be interpreted as the anticlockwise angle between $z=ie^{u+iv}$ and the positive imaginary axis.

With respect to Huber coordinates, it is easy to see that
\begin{equation} \label{rop} K_m=e^{-iv}\cos{v}\left(\frac{\partial}{\partial u}-i\frac{\partial}{\partial v}\right)+m, \end{equation}
and
\begin{equation}\label{meshub} \, d\mu(z)=\frac{1}{\cos^2{v}}\, du\, dv. \end{equation}
The following two series over cosets in $\Gamma_1 \backslash \Gamma$ are of high importance in our work:

\begin{align}
&A^{(0)}_f(z):=\sum_{\gamma \in \Gamma_1 \backslash \Gamma} f\left( \frac{1}{ \cos^2\left( v\left( \gamma z \right) \right)} \right), \label{hseries1} \\
&A^{(1)}_f(z):=\sum_{\gamma \in \Gamma_1 \backslash \Gamma} \tan{\left(v\left( \gamma z \right) \right)}f\left( \frac{1}{ \cos^2\left( v\left( \gamma z \right) \right)} \right). \label{hseries2}
\end{align}
These are well-defined elements of $L^2{\left(\Gamma \backslash \H \right)}$ when, for example, $f$ is a continuous function with exponential decay, defined in the interval $[1,+\infty)$. This follows from \cite[Thm 1.1]{chatzakos} and partial summation. For our purposes, it is enough to consider $f$ having compact support. We make the additional assumption that $f$ is piecewise differentiable.

In the following sections, we will use the spectral expansions of these series to prove the following Theorem, which is equivalent to Theorem \ref{mainthm}.
\begin{theorem} \label{equivmain}
For any $i \in \left\{1,2,3,4\right\}$,

$$N_i\left(X \right)=\frac{2\delta_{1i}\left(\mathrm{len}(l)\right)^2}{\pi\mathrm{Vol}\left(\Gamma \backslash \mathbb{H}\right)}X+\sum_{1/2<s_j < 1}c_i(s_j)X^{s_j}+O\left(X^{2/3}\right),$$
where \begin{align*}
N_1\left(X\right):=&\sum_{\gamma: |B(\gamma)| \leq X}1,\\
N_2\left(X\right):=&\sum_{\gamma: |B(\gamma)| \leq X}\mathrm{sign}(ab), \\
N_3\left(X\right):=&\sum_{\gamma: |B(\gamma)| \leq X} \mathrm{sign}(ac), \\
N_4\left(X\right):=&\sum_{\gamma: |B(\gamma)| \leq X} \mathrm{sign}(ad)
,
\end{align*}
with $a,b,c,d$ being the entries of $\gamma$ in that order, $B(\gamma)=ad+bc$, and the summations being over the double cosets $\Gamma_1 \backslash \Gamma \slash \Gamma_1 - \left\{ \mathrm{id}\right\},$ and $$c_1(s_j)=D(s_j)  \hat{u}_{0,j}^2, \quad c_2(s_j)=D(s_j) a_j \hat{u}_{0,j}  \hat{u}_{1,j}, \quad c_3(s_j)=-c_2(s_j), \quad c_4(s_j)=-D(s_j)a^{2}_j  \hat{u}_{1,j}^2,$$
where the quantities $D(s_j)$ and $a_j$ are defined as in Theorem \ref{mainthm}.
\end{theorem}
The equivalence of Theorems $\ref{mainthm}$ and $\ref{equivmain}$ is given by
\begin{equation}\label{ndeltaequiv}N^{\mu,\mu'}\left(X\right)=\frac{N_1\left(X\right)+\mu\cdot N_2\left(X\right)+\mu'\cdot N_3\left(X\right)+\mu\mu'\cdot N_4\left(X\right)}{4}.\end{equation}

\begin{remark}
    Note that $ad-bc=1$ gives $B(\gamma)=ad+bc=2ad-1$ and, hence, $B(\gamma)$ has the same sign with $ad$ unless $ad \in [0,1/2]$. Therefore, we have $$N_4\left(X\right)=\sum_{\gamma: |B(\gamma)| \leq X} \mathrm{sign}\left( B(\gamma)\right) +O(1).$$
\end{remark}
\begin{remark}\label{th1rem1}
Using the bijection $\gamma \leftrightarrow \gamma^{-1}$, we can see that $N_{2}\left(X \right)=-N_{3}\left(X \right)+O(1)$. Furthermore, we note that $N_1(X)=N(X,l)$ as in Theorem \ref{lekkasmain}. Therefore, it is enough to consider only $N_3(X)$ and $N_4(X)$.
\end{remark}

\begin{remark}
If $\Gamma$ contains an element with both diagonal entries equal to zero, say $\gamma'$, consider the bijection $\gamma \leftrightarrow \gamma' \cdot \gamma $.  For $|B(\gamma)|>1$, this gives a bijection between the cases $\hbox{sign}(ab)>0$ and $\hbox{sign}(ab)<0$, and similarly for $\hbox{sign}(ad)$.  Hence, we will have  $N_2(X),N_3(X),N_4(X)=O(1)$. On the other hand, as $l$ is invariant under $\gamma'$ and, on $l$, we have $u_{1,j}(\gamma' z)=-u_{1,j}( z)$, all the periods $\hat{u}_{1,j}$ are identically $0$. Hence, in that case, Theorem \ref{mainthm} follows directly from Theorem \ref{lekkasmain}. Therefore, we will assume that no such element exists.
\end{remark}
\begin{remark}
    In the notation of Hejhal (see \cite[Thm.1]{hejhal1}), the series $D_{a}$ corresponds to our $N_1$, the series $D_{b}$ corresponds to $N_4$, and the series $D_{c}$, $D_{d}$ correspond jointly to our $N_2$, $N_3$.
\end{remark}

Assuming without loss of generality that $l$ is contained in the positive imaginary axis, we notice that (see \cite[Lemma 1]{voskoulekkas}), for
$$I(\phi):=\int_{l}K(z,e^{i \phi} w) ds(w),$$
 we have $$I(0)=A_{g}^{(0)}(z), \quad I'(0)=A_{h}^{(1)}(z),$$
where

\begin{align*}g(p)&:=\int_{0}^{+\infty}k\left(\frac{p}{4r}+\frac{r}{4}-\frac{1}{2} \right) \frac{dr}{r}, \nonumber \\
h(p)&:=\frac{1}{2}\int_{0}^{+\infty}k'\left(\frac{p}{4r}+\frac{r}{4}-\frac{1}{2} \right) \frac{dr}{r}. \nonumber \end{align*}
The series $A^{(0)}_f(z)$ and its spectral expansion have been studied by Huber \cite{huber}, derived in a different fashion. Huber \cite{huber} and Chatzakos--Petridis \cite{chatzakos} used this series to study the elliptic-hyperbolic case. 
 In \cite{lekkas}, Lekkas considers the integral $$I_{f,0}:=\int_{l} A_f^{(0)}(z) ds$$
to study the sum $N_1(X)$. 

In a similar manner, we can relate the quantities $N_2(X),N_3(X),N_4(X)$ with $I'_{f,0}(0)$, $I_{f,1}(0)$ and $I'_{f,1}(0)$ correspondingly for appropriate choices of the function $f$, where
\begin{equation} \label{iintdef}I_{f,j}(\theta):=\int_{l} A_f^{(j)}(e^{i \theta}z) ds. \end{equation}
We note that, for $$I_{f}(\theta,\phi):=\int_{l}\int_{l}K(e^{-i \theta}z,e^{i \phi} w) ds(w) ds(z),$$ and appropriate choices of $k$ and $f$, 
$I'_{f,0}(0)$ corresponds to the partial derivative
$ \partial_\theta I_f(0,0)$, 
$I_{f,1}(0)$ corresponds to the partial derivative
$\partial_\phi I_f(0,0)$, and, finally, $I'_{f,1}(0)$ corresponds to the mixed partial derivative
$\partial^2_{\phi \, \theta} I_f(0,0)$. 
By the symmetry of $K(z,w)$, we re-obtain the relation $N_{2}\left(X \right)=-N_{3}\left(X \right)+O(1).$

The correspondence between $N_2(X),N_3(X),N_4(X)$ and
$\partial_{\theta} I_f(0,0)$, $\partial_{\phi} I_f(0,0)$, $ \partial^2_{\phi \, \theta} I_f(0,0)$
can be demonstrated geometrically in the following way (see Figure \ref{figurefourc2}): We denote the segment $e^{i \omega}l$ by $l_\omega$. We consider how the distance between $l_{\phi}$ and $\gamma l_{\theta}$ changes for $\theta, \phi$ close to $0$, for different choices of $\gamma$ with $|B(\gamma)|>1$. We observe that for positive $\phi$, $l_{\phi}$ moves to the left. Therefore the distance from $\gamma l$ decreases for $ac<0$ but increases for $ac>0$ (recall that in the former case $\gamma l$ is in the left quadrant, while in the latter case $\gamma l$ is in the right quadrant - see Figure \ref{figurefourc}). Equivalently, the sign of the derivative of $d(l_{\phi},\gamma l)$ at $\phi=0$ is $\mu'=\hbox{sign}(ac)$. 

Similarly, we can see that, for $\theta$ positive, $\gamma l_\theta$ becomes larger (in terms of, say, Euclidean area enclosed above the real axis) if $ab>0$ but smaller if $ab<0$. In other words, it moves closer to or further from $l$ accordingly. To see why this is the case, we note that the sign of $ab$ corresponds to whether the region $\hbox{Re}(z)<0$ maps to the inside or the outside of the region enclosed by $\gamma l$ and the real axis. Hence, when $l_{\phi}$ moves to left, $\gamma l_{\phi}$ moves towards the inside or the outside of the corresponding region accordingly. Therefore, the sign of the derivative of $d(l,\gamma l_{\theta})$ at $\theta=0$ is $\mu=\hbox{sign}(ab)$. 

Combining these two observations, we expect that the sign of the mixed second order derivative of $d(l_{\phi},\gamma l_{\theta})$ with respect to both variables at $(0,0)$ should have sign $\mu \cdot \mu'$=$\hbox{sign}(ad)$.
\begin{figure}[h]\label{figurefourc2} \centering
\includegraphics[scale=0.70]{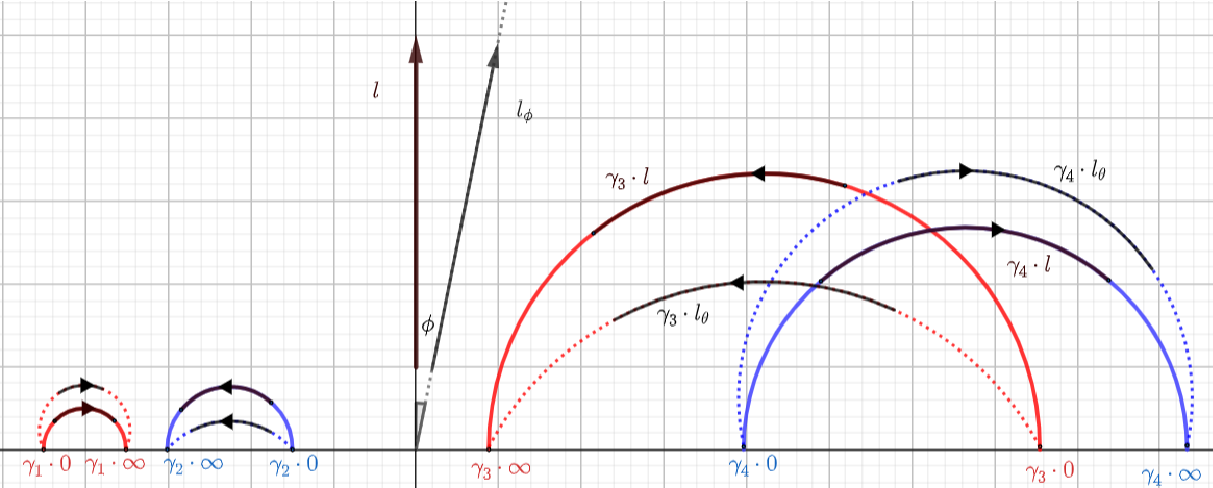}
\caption{Effect of perturbation of $l_1$ and $l_2$ on Figure 1. }
\end{figure}
\section{Spectral Expansion} \label{sesection}
The spectral expansion of the series $A_f^{(0)}(z)$ is a central element in the work of Huber and Chatzakos--Petridis in the hyperbolic-ellptic case, as well as in the work of Lekkas on the hyperbolic-hyperbolic case. In a similar manner, we will make use of the spectral expansion of $A^{(1)}_f(z)$. In particular, we will use the following result.
\begin{lemma}\label{spside} For $f$ a continuous, piecewise differentiable function with exponential decay,
we have the following spectral expansion:\label{a1spectral}
    \begin{align*}A_f^{(1)}(z)=2&\sum_{j}\sqrt{\lambda_j}d^{(1)}_{t_j}(f)\hat{u}_{1,j}u_{0,j}(z) \\ -&\sum_{\mathfrak{a}}\frac{i}{2\pi}\int_{\left(1/2\right)}\sqrt{s(1-s)} \, \cdot d_t^{(1)}(f)\hat{E}_{\mathfrak{a},1}\left(s\right)\overline{E_{\mathfrak{a},0}\left(z,s\right)}ds,\end{align*}
    where
\begin{equation} \label{d1def} d^{(1)}_{t}(f):=\int_{0}^{\pi/2}\frac{\tan^2{v}}{\cos^2{v}} f\left(\frac{1}{\cos^2{v}}\right) \cdot {}_2F_1\left(\frac{s+1}{2},\frac{2-s}{2} \, ; \frac{3}{2} \, ; -\tan^{2}{v} \right)\, dv,\end{equation}
and ${}_{p}F_q$ is the (generalized) hypergeometric function (see Appendix \ref{spapp}). 
\end{lemma}
\begin{proof} For convenience, we only demonstrate the cocompact case. For the general cofinite case, the only difference is that we have to take into account the contribution of the Eisenstein series, which can be treated in a similar manner.

We have that $$A^{(1)}_f(z)=\sum_{j}c^{(1)}_{j}(f)u_{0,j}(z),$$
where
\begin{align*}c^{(1)}_{j}(f)&=\int_{\Gamma \backslash \mathbb{H}}A^{(1)}_f(z)\overline{u_{0,j}(z)}\, d\mu(z)\\
&=\sum_{\gamma \in \Gamma_1 \backslash \Gamma}\int_{\Gamma \backslash \mathbb{H}}\tan{\left(v\left( \gamma z \right) \right)}f\left( \frac{1}{ \cos^2\left( v\left( \gamma z \right) \right)} \right)\overline{u_{0,j}(z)}\, d\mu(z).  \end{align*}
Using equation (\ref{meshub}), we can rewrite this as
\begin{equation*}
c^{(1)}_{j}(f)=\int_{0}^{\mathrm{len}(l)}\int_{-\pi/2}^{\pi/2}\frac{\tan{v}}{\cos^2{v}}f\left(\frac{1}{\cos^2{v}} \right)\overline{u_{0,j}(ie^{u+iv})} dv \, du.  \end{equation*}
Hence,
\begin{equation*}
c^{(1)}_{j}(f)=\int_{0}^{\pi/2}\frac{\tan{v}}{\cos^2{v}}f\left(\frac{1}{\cos^2{v}} \right)\overline{V_{j}(v)} dv,\end{equation*}
where $$V_{j}(v)=U_{j}(v)-U_{j}(-v),$$
and $U_j(v):=\int_{l}u_{0,j}(z)du$ is a solution of $$ F^{''}+\frac{\lambda_i}{\cos^2{v}}F=0.$$
In other words, $V$ is a solution of the above equation with $F(0)=0$ and $F'(0)=2\sqrt{\lambda_j}\hat{u}_{1,j}$. Here, we used the fact that $U_j'(0)=\sqrt{\lambda_j}\hat{u}_{1,j}$, which follows from equation (\ref{u0tou1}). Similarly with Chatzakos--Petridis \cite[Section~2.2]{chatzakos}, for suitable coefficients $a(s)$ and $b(s)$, we can write \begin{displaymath}
V_j(v) = a(s){}_2 F_1\left(s,1-s;1; \frac{1 - i \tan(v)}{2}\right) + b(s)  {}_2 F_1\left(s,1-s;1; \frac{1 + i \tan(v)}{2} \right).
\end{displaymath}
The initial conditions imply that 
$a(s) =- b(s)$. Applying equation (\ref{quadtrans}), we have
$$V_{j}(v)= c(s) \cdot  \tan{v} \cdot  {}_2F_1\left( \frac{s+1}{2}, \frac{2-s}{2} \, ; \frac{3}{2} \, ; -\tan^2{v} \right), $$
for some coefficient $c(s)$ that depends only on $s$.
Using the initial conditions once again, we conclude that
$$V_{j}(v)= 2\sqrt{\lambda_j}\hat{u}_{1,i} \cdot  \tan{v} \cdot  {}_2F_1\left( \frac{s+1}{2}, \frac{2-s}{2} \, ; \frac{3}{2} \, ; -\tan^2{v} \right). $$ \end{proof}
The following Lemma provides some useful alternative formulae for the transform
$d^{(1)}_t(f)$.
\begin{lemma} \label{altformulas} For $f$ a continuous, piecewise differentiable function with exponential decay and $s=1/2+it\in\mathbb{C}$, the two following formulas hold:
   \begin{align*} &\textup{(i)}& d^{(1)}_{t}(f)
&=\int_{0}^{+ \infty}x^2 f\left(1+x^2 \right) \cdot {}_2F_1\left( \frac{s+1}{2}, \frac{2-s}{2} \, ; \frac{3}{2} \, ; -x^2 \right)dx,  &\qquad& \\  
&\textup{(ii)}&
 d^{(1)}_{t}(f)&=-\int_{0}^{+ \infty}\frac{x^2}{2} \cdot\left(x f\left(1+x^2 \right)\right)' \cdot  {}_3F_2\left(1, \frac{s+1}{2}, \frac{2-s}{2} \, ; 2, \frac{3}{2} \, ; -x^2 \right) \, dx. &\qquad&
\end{align*}

\end{lemma}

\begin{proof}
\,
 \begin{enumerate}[(i)]
            \item Apply the substitution $x=\tan{v}$ in equation (\ref{d1def}).
            \item We write
\begin{align}d^{(1)}_{t}(f)&=\int_{0}^{+ \infty}x^2 f\left(1+x^2 \right) \cdot {}_2F_1\left( \frac{s+1}{2}, \frac{2-s}{2} \, ; \frac{3}{2} \, ; -x^2 \right)dx \nonumber \\
&=\int_{0}^{+ \infty}\left(x f\left(1+x^2 \right)\right)\left(x \cdot {}_2F_1\left( \frac{s+1}{2}, \frac{2-s}{2} \, ; \frac{3}{2} \, ; -x^2 \right) \right) dx. \label{dtbyparts} \end{align}
We note that, using the substitution $y=x^2u$ and equation (\ref{hintrans}), we have that
\begin{align*} \int_{0}^{x}y \cdot {}_2F_1\left( \frac{s+1}{2}, \frac{2-s}{2} \, ; \frac{3}{2} \, ; -y^2 \right)  dy&= \int_{0}^{1}\frac{x^2}{2} \cdot {}_2F_1\left( \frac{s+1}{2}, \frac{2-s}{2} \, ; \frac{3}{2} \, ; -x^2u \right) du \\ &=\frac{x^2}{2} \cdot {}_3F_2\left(1, \frac{s+1}{2}, \frac{2-s}{2} \, ; 2, \frac{3}{2} \, ; -x^2 \right).\end{align*}
Hence, using integration by parts on equation (\ref{dtbyparts}), we conclude that

 \begin{eqnarray}d^{(1)}_{t}(f)
&=&-\int_{0}^{+ \infty}\left(x f\left(1+x^2 \right)\right)' \frac{x^2}{2} \cdot {}_3F_2\left(1, \frac{s+1}{2}, \frac{2-s}{2} \, ; 2, \frac{3}{2} \, ; -x^2 \right) \, dx.
\nonumber\end{eqnarray}
    \end{enumerate}
\end{proof}

\section{Modified Relative-Trace Formulae} \label{mrtfsection}
We now introduce modified relative trace formulae related to $\Gamma_1 \backslash \Gamma \slash \Gamma_1$, which will be crucial for our argument.

\begin{theorem}[Modified Relative Trace Formulae] \label{modtrace}Let $f$ be a real, continuous, piecewise differentiable function with exponential decay. Let $\varepsilon$ be equal to $1$ if $\Gamma$ has an element with both diagonal entries being equal to zero, and $0$ otherwise. 
We define
$$f_{(a)}:=f,\quad
 f_{(b)}:=\sqrt{x-1}\cdot f,\quad f_{(c)}:=f+2\sqrt{x-1} \cdot f'.$$
 We further define $g^{+}_0=(1+\varepsilon)f(1)\mathrm{len}(l)$ and $g^{-}_0=(1-\varepsilon)f(1)\mathrm{len}(l)$.
 Then, we have
\begin{align*}
\textup{(a) }& \; \displaystyle &g^{+}_0+\sum_{\gamma \in \Gamma_1 \backslash \Gamma \slash \Gamma_1-\left\{ \mathrm{id} \right\}} g\left(B(\gamma)^2;f_{(a)}\right)&=&\phantom{-}2&\sum_{j} d^{(0)}_{t_j}(f)\hat{u}^2_{0,j}+E^{(a)}(f),\\
\textup{(b) }& \; \displaystyle  &\sum_{\gamma \in \Gamma_1 \backslash \Gamma \slash \Gamma_1: |B(\gamma)|>1}\mathrm{sign}(ac) \cdot g\left(B(\gamma)^2;f_{(b)}\right) &=&-2&\sum_{j}\lambda_j^{1/2} d^{(1)}_{t_j}(f)\hat{u}_{1,j}\hat{u}_{0,j}+E^{(b)}(f),\\
\textup{(c) }& \; \displaystyle &g^{-}_0+\sum_{\gamma \in \Gamma_1 \backslash \Gamma \slash \Gamma_1 - \left\{ \mathrm{id} \right\}} B(\gamma) \cdot g\left(B(\gamma)^2; f_{(c)}\right)&=&\phantom{-}2&\sum_{j} \lambda_jd^{(1)}_{t_j}(f)\hat{u}^2_{1,j}+E^{(c)}(f),
\end{align*}
where $B(\gamma)= ad+bc$, and, for $s=1/2+it$,
\begin{align*}d^{(0)}_{t}(f)&:=\int_{0}^{\pi/2} \frac{1}{\cos^2{v}}f\left(\frac{1}{\cos^2{v}}\right)\cdot {}_2F_1\left(\frac{s}{2},\frac{1-s}{2} \, ; \frac{1}{2} \, ; -\tan^{2}{v} \right)\, dv, \\
d^{(1)}_{t}(f)&:=\int_{0}^{\pi/2}\frac{\tan^2{v}}{\cos^2{v}} f\left(\frac{1}{\cos^2{v}}\right) \cdot {}_2F_1\left(\frac{s+1}{2},\frac{2-s}{2} \, ; \frac{3}{2} \, ; -\tan^{2}{v} \right)\, dv,\end{align*}
and 
\begin{equation} \label{gdeff} g(u;h):=2\int_{ \sqrt{\mathrm{max}(u-1,0)}}^{\infty}  \frac{h\left( x^2+1  \right)}{\sqrt{x^2+1-u}}dx=\int_{ \mathrm{max}(u,1)}^{\infty}  \frac{h\left( v  \right)}{\sqrt{v-u}}\frac{dv}{\sqrt{v-1}}, \end{equation}
and \begin{align*}E^{(a)}(f):=-&\sum_{\mathfrak{a}} \frac{i}{2 \pi}  \int_{\left(1/2\right)} \di  \left|\hat{E}_{\mathfrak{a},0}\left(s \right) \right|^2  
\,ds, \\
E^{(b)}(f):=\phantom{-}&\sum_{\mathfrak{a}} \frac{i}{2 \pi}  \int_{\left(1/2\right)} (s(1-s))^{1/2}\dii  \hat{E}_{\mathfrak{a},1}\left(s \right)  \overline{\hat{E}_{\mathfrak{a},0}\left(s \right) }  
\,ds, \\
E^{(c)}(f):=-&\sum_{\mathfrak{a}} \frac{i}{2 \pi}  \int_{\left(1/2\right)} s(1-s)\dii  \left|\hat{E}_{\mathfrak{a},1}\left(s \right) \right|^2  
\,ds. \end{align*}
\end{theorem}
\begin{remark}
    The proof of (a) for $\Gamma$ cocompact and $\varepsilon=0$ can be found in \cite[\S 3.1]{lekkas}.
\end{remark}
\begin{remark}
    We note that, while $u_{1,j}(z)$ is not necessarily real-valued over $\mathbb{C}$, the periods $\hat{u}_{1,j}$ always are. Indeed, for $(u,v)$ the Huber coordinates described in equation (\ref{coordsys}), we have $$ u_{1,j}(z):=\frac{i}{\sqrt{\lambda_j}}K_0u_{0,j}(z)=\frac{ie^{-iv}}{\sqrt{\lambda_j}}\cos{v}  \left(\frac{\partial}{\partial u}-i\frac{\partial}{\partial v} \right)u_{0,j}(z).$$ 
    On the other hand, $u_{0,j}(z)$ is periodic with respect to the parameter $u$, as it is by definition invariant under the action of $\Gamma$, and $\Gamma$ contains a diagonal element. Hence, using the fact that $v(z)=0$ on the geodesic segment $l$, we have that
    \begin{equation} \label{u0tou1}\hat{u}_{1,j}:=\int_{l}u_{1,j}(z) ds(z)=\int_{0}^{\text{len}(l)}u_{1,j}(z) \, du =\lambda_j^{-1/2} \int_{0}^{\text{len}(l)}  \frac{\partial}{\partial v} u_{0,j}(z)\, du.\end{equation} As  $u_{0,j}(z)$ is real\hyp{}valued, we conclude that $\hat{u}_{1,j}$ is real as well. Therefore, we do not have to conjugate the second factor in the spectral expansions.
\end{remark}

\begin{remark}
    It is worth noting that, in the case $\varepsilon=1$, the second and third part of the theorem are trivial, as both sides of the equations are identically $0$. We can see this by considering the bijection
    $\gamma \leftrightarrow \gamma'\gamma$, where $\gamma'$ has both diagonal entries equal to zero.
\end{remark}
\begin{remark}
    For $|B(\gamma)|>1$, the quantity $\cosh^{-1}{B(\gamma)}$ is the hyperbolic distance of $l$ from $\gamma \cdot l$. The case $|B(\gamma)| \leq 1$ corresponds to the cases where $l$ and $\gamma \cdot l$ intersect. See for example \cite[Lemma 1]{mckee}.
\end{remark}

For the proof of the modified relative trace formula, it suffices to combine Proposition \ref{geoside}, which deals with the left hand side (\emph{geometric side}) of the trace formulae, with the spectral expansion of $A^{(1)}_f(z)$ (Lemma \ref{spside}), which deals with the right hand side (\emph{spectral side}).
For simplicity, we only demonstrate the case where $\Gamma$ is cocompact and $\varepsilon=0$ (where $\varepsilon$ is as in Theorem \ref{modtrace}). The general case is similar.
\begin{proposition} \label{geoside}
Assume that $\Gamma$ is cocompact and $\varepsilon=0$. For $f$ a continuous, piecewise differentiable function with exponential decay, $I_{f,1}(\theta)$ as defined in equation (\ref{iintdef}), and $f_{(b)}, f_{(c)}$, and $g$ as in Theorem \ref{modtrace}, we have
\begin{alignat*}{5}
 &\hbox{\textup{(i)} } \quad && \qquad \displaystyle
&I_{f,1}(0)\;
&=-&&\sum_{\substack{\gamma \in \Gamma_1 \backslash \Gamma \slash \Gamma_1: \\ \left|B(\gamma)\right|>1}}\mathrm{sign}(ac) \cdot g\left(B(\gamma)^2;f_{(b)}\right), &\qquad  &
 \\
&\hbox{\textup{(ii)} } \quad && \qquad \displaystyle
&I'_{f,1}(0)\;&=&&\sum_{\gamma \in \Gamma_1 \backslash \Gamma \slash \Gamma_1-\left\{ \mathrm{id} \right\}}B(\gamma) \cdot g\left(B(\gamma)^2; f_{(c)}\right) +f(1)\mathrm{len}\left(l\right). & \qquad &
\end{alignat*}
\end{proposition}
\begin{proof} For $w=e^{i \theta}z$, we have 
        \begin{equation*}I_{f,1}(\theta)=\int_{l}A^{(1)}_f\left(w \right) ds(z)
        =  \sum_{\gamma \in \Gamma_1 \backslash \Gamma} \int_{l} \tan{\left(v\left( \gamma \cdot  w \right)  \right)}f\left( \frac{1}{ \cos^2\left( v\left( \gamma \cdot w \right) \right)} \right) ds(z). 
 \end{equation*}
 First, we consider the term corresponding to $\gamma=\mathrm{id}$, the identity class in $\Gamma_1 \backslash \Gamma$, separately. 
 We have 
\begin{align*} \int_{l} \tan{\left(v\left( \mathrm{id} \cdot  w \right)  \right)}f\left( \frac{1}{ \cos^2\left( v\left( \mathrm{id} \cdot w \right) \right)} \right) ds(z)&=\int_{l} \tan{\left(\theta  \right)}f\left( \frac{1}{ \cos^2\left( \theta \right)} \right) ds(z) \\
&=\tan{\theta} \cdot f\left( \frac{1}{\cos^2{\theta}}\right)\mathrm{len}(l). \end{align*}
For the rest of the cosets, i.e., $ \Gamma_1 \backslash \Gamma -\left\{ \mathrm{id} \right\}$, we note that, for any given $\gamma \neq \mathrm{id}$, the cosets corresponding to $\gamma \gamma_0$, where $\gamma_0$ runs through the elements of $\Gamma_1$, are disjoint. This follows from the assumption $\varepsilon=0$. Indeed, otherwise, we will have that some $\tilde{\gamma} \in \Gamma - \Gamma_1$ will satisfy $\gamma_0' \tilde{\gamma}= \tilde{\gamma} \gamma_0$, for some elements $\gamma_0, \gamma'_0 \in \Gamma_1$, with $\gamma_0 \neq \mathrm{id}$. It is easy to check that this happens if any only if $\gamma$ has both diagonal entries equal to $0$, giving, by definition, $\varepsilon=1$. Therefore, we can proceed as follows:
        \begin{align*}
        &  \sum_{\gamma \in \Gamma_1 \backslash \Gamma -\left\{ \mathrm{id} \right\}} \int_{l} \tan{\left(v\left( \gamma \cdot  w \right)  \right)}f\left( \frac{1}{ \cos^2\left( v\left( \gamma \cdot w \right) \right)} \right) \, ds(z)  \\ &=\sum_{\gamma \in \Gamma_1 \backslash \Gamma \slash \Gamma_1-\left\{ \mathrm{id} \right\}} \sum_{\gamma_0 \in \Gamma_1} \int_{l} \tan{\left(v\left( \gamma \gamma_0 w \right)  \right)}f\left( \frac{1}{ \cos^2\left( v\left( \gamma \gamma_0 w \right) \right)} \right) \, ds(z)
        \\
        &=\sum_{\gamma \in \Gamma_1 \backslash \Gamma \slash \Gamma_1-\left\{ \mathrm{id} \right\}} \int_{0}^{+\infty} \tan{\left(v\left( \gamma \cdot e^{i \theta} i y \right) \right)}f\left( \frac{1}{ \cos^2\left( v\left( \gamma \cdot e^{i \theta} i y \right) \right)} \right) \frac{dy}{y}.
 \end{align*}

We denote the summand by \begin{equation}I(\theta; \, \gamma):=\int_{0}^{+\infty} \tan{\left(v\left( \gamma \cdot e^{i \theta} i y \right) \right)}f\left( \frac{1}{ \cos^2\left( v\left( \gamma \cdot e^{i \theta} i y \right) \right)} \right) \frac{dy}{y}, \label{termstodiff} \end{equation}
so that 
\begin{equation} \label{bandc} I_{f,1}(\theta)=\tan{\theta} \cdot f\left( \frac{1}{\cos^2{\theta}}\right) \cdot \mathrm{len}(l)+\sum_{\gamma \in \Gamma_1 \backslash \Gamma \slash \Gamma_1-\left\{ \mathrm{id} \right\}} I(\theta; \, \gamma). \end{equation}
For (ii) we compute the derivative of equation (\ref{bandc}) at $\theta=0$.
In particular, $$I'_{f,1}(0)=f\left( 1\right) \cdot \mathrm{len}(l)+\sum_{\gamma \in \Gamma_1 \backslash \Gamma \slash \Gamma_1-\left\{ \mathrm{id} \right\}} I'(0; \, \gamma).$$
We note that
\begin{equation*}  \tan\left( v(\gamma e^{i \theta}iy) \right)=B(\gamma)\tan{\theta}-\frac{acy+bd/y}{\cos{\theta}}, \end{equation*}
and, hence,
\begin{equation}
\left. \frac{\partial}{\partial \theta}  \tan\left( v(\gamma e^{i \theta}iy) \right) \right \vert_{\theta=0}=B(\gamma). \label{tanest} \end{equation}
Differentiating (\ref{termstodiff}) and using equation (\ref{tanest}), we arrive at
$$I'(0; \, \gamma)=B(\gamma)\int_{0}^{\infty}  f\left( \left( \frac{bd}{y}+acy \right)^2+1  \right)+2\left( \frac{bd}{y}+acy \right)^2f'\left( \left( \frac{bd}{y}+acy \right)^2+1  \right) \frac{dy}{y}.$$
For the case of double cosets with $abcd>0$, we use the substitution $x=|bd|/y+|ac|y$. This gives
$$dx = \left(-|bd|/y^2+|ac| \right) dy,$$
i.e.,
$$ \frac{dy}{y}=\pm \frac{dx}{\sqrt{x^2-4abcd}},$$
where the sign is positive for $y>\sqrt{\left|bd/ac\right|}$ and negative otherwise. Therefore, we have
\begin{equation*}I'(0; \,\gamma)=
B(\gamma) \cdot \int_{ \sqrt{M(\gamma)}}^{\infty} 2\left(x f\left( x^2+1  \right)\right)' \frac{dx}{\sqrt{x^2-\left(B(\gamma)^2-1 \right)}}, 
\end{equation*}
where $M(\gamma)=\hbox{max}\left( B(\gamma)^2-1,0\right)$.
For $abcd<0$, we take $x=-|bd|/y+|ac|y$. This gives
\begin{align*} I'(0; \,\gamma)
&=B(\gamma) \cdot \int_{ -\infty}^{\infty} \left(x f\left( x^2+1  \right)\right)' \frac{dx}{\sqrt{x^2-\left(B(\gamma)^2-1 \right)}}  \\
&=B(\gamma) \cdot\int_{0}^{\infty} 2\left(x f\left( x^2+1  \right)\right)' \frac{dx}{\sqrt{x^2-\left(B(\gamma)^2-1 \right)}}.
\end{align*}
Finally, we note that there are no non-trivial cosets with $abcd=0$, as this would give $\varepsilon=1$ by the discreteness of the group $\Gamma$.
Hence,
\begin{equation*}
I'_{f,1}(0)
=2 \! \! \! \! \! \sum_{\gamma \in \Gamma_1 \backslash \Gamma \slash \Gamma_1-\left\{ \mathrm{id} \right\}} \! \! \! \! \! B(\gamma) \int_{ \sqrt{M(\gamma)}}^{\infty} \left(x f\left( x^2+1  \right)\right)' \frac{dx}{\sqrt{x^2-\left(B(\gamma)^2-1 \right)}}+f(1)\mathrm{len}(l).
\end{equation*}
For (i), we plug $\theta=0$ in equation (\ref{bandc}) and proceed in a similar fashion:
\begin{align*}
I_{f,1}(0)&=-\sum_{\gamma \in \Gamma_1 \backslash \Gamma \slash \Gamma_1}\int_{0}^{\infty} \left( \frac{bd}{y}+acy \right) f\left( \left( \frac{bd}{y}+acy \right)^2+1  \right) \frac{dy}{y} \\ 
&=-2 \sum_{\gamma \in \Gamma_1 \backslash \Gamma \slash \Gamma_1: |B(\gamma)|>1}\mathrm{sign}(ac) \int_{ \sqrt{M(\gamma)}}^{\infty} x f\left( x^2+1  \right) \frac{dx}{\sqrt{x^2-\left(B(\gamma)^2-1 \right)}}.
\end{align*}
\end{proof}

 To finish the proof of Theorem \ref{modtrace}, we substitute the spectral expansion of $A^{(1)}_f$ (Lemma \ref{a1spectral}) in the definition of $I_{f,1}(\theta)$, to get:
$$ I_{f,1}(\theta)=\int_{l}A^{(1)}_f\left(w \right) \, ds(z) =2\sum_{j}\sqrt{\lambda_j}d^{(1)}_{t_j}(f)\hat{u}_{1,j}\int_{l}u_{0,j}(w) \,ds(z).$$
This gives
\begin{equation} \label{if1} I_{f,1}(0)=2\sum_{j}\sqrt{\lambda_j}d^{(1)}_{t_j}(f)\hat{u}_{1,j}\hat{u}_{0,j}, \end{equation}
and, using equation (\ref{u0tou1}),
\begin{equation} \label{if2} I'_{f,1}(0)=-2\sum_{j}\lambda_jd^{(1)}_{t_j}(f)\hat{u}^2_{1,j}. \end{equation}
Theorem \ref{modtrace} now follows by combining Proposition \ref{geoside} with equations (\ref{if1}) and (\ref{if2}).

\section{Choice for Test Functions} \label{choicesection}
In light of Theorem \ref{modtrace}, in order to approximate $N_1,N_3$ and $N_4$, we want to choose $f_1$, $f_3$, $f_4$ so that
\begin{eqnarray}
& &g(u,f_1) \sim \mathbbm{1}_{[0,X^2]}(u), \nonumber \\
& &g(u,\sqrt{x-1} \cdot f_3) \sim \mathbbm{1}_{[0,X^2]}(u), \nonumber \\
& &g\left(u,f_4+2\sqrt{x-1} \cdot f'_4\right) \sim \mathbbm{1}_{[0,X^2]}(u)\cdot \frac{1}{\sqrt{u}}, \nonumber
\end{eqnarray}
where, for $A \subset \mathbb{R}$,  $\mathbbm{1}_{A}(u)$ denotes the indicator function of the set $A$.

 As in the work of Lekkas \cite{lekkas}, a valid choice for $f_1$ is the following:
 \begin{equation*} f_1(x^2+1) = \left\{
\begin{array}{ll}
      \displaystyle 2 \pi H^{-1} x \cdot  \left( \sqrt{R^2-x^2}-\sqrt{r^2-x^2} \right), & 0\leq x \leq r, \\
      \displaystyle 2 \pi H^{-1}x \cdot \sqrt{R^2-x^2}, & r \leq x \leq R,\\
      \displaystyle 0, & R \leq x, \\
\end{array} 
\right. \end{equation*}
where \begin{equation} \label{rrdef} R^2=(X+Y)^2-1, \quad r^2=X^2-1, \quad H=R^2-r^2, \end{equation}
and $Y=D\cdot X$, where $0<D<1$ is independent of $X$.
We note that $f_1(x)/\sqrt{x-1}$ still satisfies the conditions of Theorem \ref{modtrace}, and, therefore, we can take \begin{equation} \label{f3def} 
f_3=f_1/\sqrt{x-1}.  \end{equation} We are left to find an appropriate function $f_4$. For technical reasons, instead of $\mathbbm{1}_{[0,X^2]}(u)/\sqrt{u}$, we consider $\mathbbm{1}_{[0,X^2]}(u)/\sqrt{u-1}$ . Fortunately, this does not affect the result. Indeed, for large $z$, we have
$$\frac{z}{\sqrt{z^2-1}}=1+ \frac{z-\sqrt{z^2-1}}{\sqrt{z^2-1}}=1+ \frac{1}{\left( z+\sqrt{z^2-1} \right) \cdot \sqrt{z^2-1} } =1 + O\left( z^{-2} \right),$$
and, therefore, we have
$$\sum_{\substack{\gamma \in \Gamma_1 \backslash \Gamma \slash \Gamma_1: \\ 1<\left|B(\gamma)\right|<X}}\frac{B\left(\gamma \right)}{\sqrt{B\left(\gamma \right)^2-1}}=\sum_{\substack{\gamma \in \Gamma_1 \backslash \Gamma \slash \Gamma_1: \\ 1<\left|B(\gamma)\right|<X}}\hbox{sign}\left( B(\gamma) \right)+ O \left( \sum_{\substack{\gamma \in \Gamma_1 \backslash \Gamma \slash \Gamma_1: \\ \left|B(\gamma)\right|<X}} \frac{1}{B^2(\gamma)} \right).$$
By \cite[Lemma 20]{tsuzukiletter},
$$\sum_{\substack{\gamma \in \Gamma_1 \backslash \Gamma \slash \Gamma_1: \\ \left|B(\gamma)\right|<X}} \frac{1}{B^2(\gamma)}= O(1).$$
Hence,
\begin{equation}\label{approx1}\sum_{\substack{\gamma \in \Gamma_1 \backslash \Gamma \slash \Gamma_1: \\ 1<\left|B(\gamma)\right|<X}}\frac{B\left(\gamma \right)}{\sqrt{B\left(\gamma \right)^2-1}}=\sum_{\substack{\gamma \in \Gamma_1 \backslash \Gamma \slash \Gamma_1: \\ 1<\left|B(\gamma)\right|<X}}\hbox{sign}\left( B(\gamma) \right)+ O \left( 1\right).\end{equation}
We will choose $f_4$ so that:
$$ g\left(u,f_4+2\sqrt{x-1} \cdot f'_4\right) = \left\{
\begin{array}{ll}
      \displaystyle au+b, & 1\leq u \leq 3, \\
     \displaystyle  \frac{1}{\sqrt{u-1}}, & 3 \leq u \leq X^2,\\
     \displaystyle  \frac{M}{\sqrt{u-1}}-B, & X^2 \leq u \leq (X+Y)^2,\\
     \displaystyle 0, & (X+Y)^2 \leq u, \\
\end{array} 
\right. $$
where $a$ will be determined later and $b,M,B$ are chosen to ensure continuity, as follows:
$$b=\frac{1}{\sqrt{2}}-3a, \quad M=
\frac{R}{R-r}, \quad  B=
\frac{1}{R-r}. $$
In particular, we note that $g$ is continuous, with compact support, and approximates $\mathbbm{1}_{[0,X^2]}(u)\cdot u^{-1/2}$, as required.

Note that, by \cite[Eq.1.64]{iwaniec}, we have that, for $v>1$
and $g(u; \, h)$ as in equation (\ref{gdeff}),
\begin{equation}
    h(v)/\sqrt{v-1}=-\frac{1}{\pi}\int_{v}^{+\infty} \frac{g'(u; \,h)}{\sqrt{u-v}} du .\label{invgtransf}
\end{equation}
From equation (\ref{invgtransf}), we have
\begin{align*} &\displaystyle \frac{f_4(v)+2\sqrt{v-1} \cdot f'_4(v)}{\sqrt{v-1}}=\nonumber \\ & \qquad \qquad
\left\{
\begin{array}{ll}
I_{X^2}(v)+M\left(I_{(X+Y)^2}(v)-I_{X^2}(v) \right)-I_{3}(v)-2a\pi^{-1}\sqrt{3-v}, & 1<v \leq 3, \\
      I_{X^2}(v)+M\left(I_{(X+Y)^2}(v)-I_{X^2}(v) \right), & 3<v \leq X^2, \\
     M\cdot I_{(X+Y)^2}(v), & X^2< v <(X+Y)^2, \\
     0, & v>(X+Y)^2,
\end{array} 
\right. \nonumber \end{align*}
where 
$$I_{r}(v)=\frac{1}{2\pi} \int_{v}^{r} \frac{1}{\sqrt{(u-v)(u-1)^3}} \, du.$$
By direct integration, we get
$$I_{r}(v)=\frac{1}{\pi \sqrt{r-1}}\frac{\sqrt{r-v}}{v-1}.$$

We deduce
\begin{equation} \label{f4def} \left(xf_4(x^2+1)\right)'= \left\{
\begin{array}{ll}
\frac{1}{\pi x(R-r)}\left( \sqrt{R^2-x^2}-\sqrt{r^2-x^2}\right)-\frac{\sqrt{2}+4ax^2}{2\pi x} \sqrt{2-x^2},& 0<x \leq \sqrt{2}, \\
      \frac{1}{\pi x(R-r) }\left( \sqrt{R^2-x^2}-\sqrt{r^2-x^2}\right),& \sqrt{2}<x \leq r, \\
     \frac{1}{\pi x(R-r) }\sqrt{R^2-x^2}, & r< x <R, \\
     0, & x>R.
\end{array}  \right. \end{equation}
For $f_4$ to be well defined and satisfy our conditions, we want to choose $a$ so that the right-hand side has an antiderivative that is $0$ at both $x=0$ and $x \geq R$. Up to translation, we can assume that the former is true independently of $a$. Hence, we want to choose $a$ so that
$$\int_{0}^{R}\left(xf_4(x^2+1)\right)' dx=0,$$

i.e., $$\frac{2a}{\pi} \int_{1}^{3} \sqrt{3-v} \, dv = \lim_{\varepsilon \rightarrow 0} \left(\frac{\left(J(R,\epsilon)-J(r,\epsilon)\right)}{\pi(R-r)} - \frac{\sqrt{2}}{2\pi}J(\sqrt{2},\epsilon)\right),$$
where, for some constant $k$, \begin{eqnarray}J(c,\epsilon):=\int_{1+\epsilon}^{c^2+1}\frac{\sqrt{c^2+1-v}}{v-1} dv&=&c\int_{\epsilon/c^2}^{1}\frac{\sqrt{1-t}}{t} dt\nonumber \\
&=&2c\tanh^{-1}\left(\sqrt{1-\frac{\epsilon}{c^2}}\right)-2c\sqrt{1-\frac{\epsilon}{c^2}}\nonumber \\
&=& kc-c\log{\frac{\epsilon}{c^2}}+O(\epsilon)=kc+2c\log{c}-c\log{\epsilon}+O(\epsilon). \nonumber \end{eqnarray}
Hence,
\begin{eqnarray}a &=& \frac{3 \sqrt{2}}{16}\lim_{\varepsilon \rightarrow 0} \left(\frac{kR+2R\log{R}-R\log{\epsilon}-kr-2r\log{r}+r\log{\epsilon}}{R-r} - \left(k+\log{2}-\log{\epsilon}\right)\right) \nonumber \\
&=& \frac{3\sqrt{2}}{8}\left( \frac{R\log{R}-r\log{r}}{R-r}-\frac{\log{2}}{2}\right).  \label{aldef} \end{eqnarray}
We now verify that our choice of test functions gives, up to a small error, the required counting function.
By Theorem \ref{lekkasmain}, we have that, for any $i \in \left\{1,2,3,4\right\}$,
\begin{equation} \label{shorteq} N_i(X+Y)-N_i(X) \ll N_1(X+Y)-N_1(X)=O(X^{2/3}+Y). \end{equation}
Using equation (\ref{shorteq}), equation (\ref{approx1}), and the definition of $f_4$, we reach the estimate \begin{equation*}
I'_{f_4,1}(0) =\sum_{\substack{\gamma \in \Gamma_1 \backslash \Gamma \slash \Gamma_1: \\ 1<\left|B(\gamma)\right|<X}}\mathrm{sign}(ad) + O \left( Y+X^{2/3} \right),\end{equation*}
where $I_{f,1}$ is as in equation (\ref{iintdef}). To summarize, we have:
\begin{eqnarray} 
I_{f_3,1}(0)&=&-\sum_{\gamma: |B(\gamma)|<X} \mathrm{sign}(ac) + O\left(Y+X^{2/3}\right), \nonumber
\\
I'_{f_4,1}(0)&=& \phantom{-}\sum_{\gamma: |B(\gamma)|<X} \mathrm{sign}(ad) + O\left(Y+X^{2/3}\right). 
\label{countingsub}  \end{eqnarray}
\section{Estimates for the Spectral Coefficients} \label{estsection}
In this section, we provide estimates for the transforms $d_{t}^{(1)}(f_3)$ and $d_{t}^{(1)}(f_4)$. In Section \ref{mainthmsection}, we will use them to conclude Theorem \ref{mainthm} from Theorem \ref{modtrace}. We start by writing the transforms in terms of generalized hypergeometric functions $J_s$ and $K_s$. Then, we make use of the known results on hypergeometric functions, which we provide in Appendix \ref{spapp}.

Applying Lemma \ref{altformulas}(i) with $f=f_3$ and Lemma \ref{altformulas}(ii) with $f=f_4$, where $f_3,f_4$ are as in equations (\ref{f3def}) and (\ref{f4def}) respectively,
we have
\begin{eqnarray}
    d^{(1)}_{t}(f_3)&=&\frac{J_s(R^2)-J_s(r^2)}{R^2-r^2}, \label{mvt2} \\
    d^{(1)}_{t}(f_4)
&=&\frac{K_s(R)-K_s(r)}{R-r}-\frac{1}{\sqrt{2}}K_s(\sqrt{2})+K_s(R,r), \label{mvt3}
\end{eqnarray}

where 
\begin{eqnarray}
J_s(u)&:=&\frac{2}{\pi}\int_{0}^{\sqrt{u}} x^2\sqrt{u-x^2} \cdot {}_2F_1\left( \frac{s+1}{2}, \frac{2-s}{2} \, ; \frac{3}{2} \, ; -x^2 \right) dx,  \label{mvtforhalf1}  \\
K_s(u)&:=&-\frac{1}{2\pi  }\int_{0}^{u} x\sqrt{u^2-x^2} \cdot {}_3F_2\left(1, \frac{s+1}{2}, \frac{2-s}{2} \, ; 2, \frac{3}{2} \, ; -x^2 \right) \, dx, \label{mvtforhalf2}
\end{eqnarray}
and
\begin{eqnarray}K_s(R,r)&:=&\frac{a}{\pi}\int_{0}^{\sqrt{2}} x^3 \sqrt{2-x^2} \cdot {}_3F_2\left(1, \frac{s+1}{2}, \frac{2-s}{2} \, ; 2, \frac{3}{2} \, ; -x^2 \right) \, dx.  \nonumber 
\end{eqnarray}
We will now evaluate these integrals using the integral hypergeometric transformation given by equation (\ref{hintrans}). For this end, we first apply the substitution $x^2=vu$ in the case of $J_s(u)$, the substitution $x^2=vu^2$ for $K_s(u)$,  and the substitution $x^2=2v$ for $K_s(R,r)$. We then apply equation (\ref{hintrans}) in the following manner:
\begin{eqnarray}J_s(u) &=& \frac{u^2}{\pi}\int_{0}^{1} \sqrt{(1-v)v} \cdot {}_2F_1\left( \frac{s+1}{2}, \frac{2-s}{2} \, ; \frac{3}{2} \, ; -uv \right) dv \nonumber \\
&=& \frac{\Gamma(3/2)^2}{\pi\Gamma(3)}u^2 \cdot {}_2F_1\left( \frac{s+1}{2}, \frac{2-s}{2} \, ; 3 \, ; -u \right) \nonumber \\
&=& \frac{ u^2 }{8} \cdot {}_2F_1\left( \frac{s+1}{2}, \frac{2-s}{2} \, ; 3 \, ; -u \right), \nonumber
\nonumber
\end{eqnarray}
\begin{eqnarray}
    K_s(u)&=& \frac{-u^3}{4\pi  }\int_{0}^{1} \sqrt{1-v}  \cdot {}_3F_2\left(1, \frac{s+1}{2}, \frac{2-s}{2} \, ; 2, \frac{3}{2} \, ; -u^2v \right) \, dv \nonumber \\
&=& 
\frac{-u^3}{6\pi  } \cdot {}_4F_3\left(1,1, \frac{s+1}{2}, \frac{2-s}{2} \, ; \frac{5}{2},2, \frac{3}{2} \, ; -u^2 \right), \nonumber
\end{eqnarray}
\begin{eqnarray}K_s(R,r)
&=& \frac{2a \sqrt{2}}{\pi}\int_{0}^{1} v \sqrt{1-v} \cdot {}_3F_2\left(1, \frac{s+1}{2}, \frac{2-s}{2} \, ; 2, \frac{3}{2} \, ; -2v \right) \, du \nonumber \\
&=& \frac{8a \sqrt{2}}{15\pi} \cdot {}_3F_2\left(1, \frac{s+1}{2}, \frac{2-s}{2} \, ;\frac{7}{2}, \frac{3}{2} \, ; -2 \right). \nonumber
\end{eqnarray}

We now express the special functions $J_s(u)$, $K_s(u)$ in terms of generalized hypergeometric functions with variables close to $0$, rather than $\infty$, as more tools are available in this neighbourhood (see Appendix \ref{spapp}). We will do this by applying equation (\ref{hgexp}). This can be done directly for the case of $J_s(u)$. On the other hand, $K_s(u)$ contains a degenerate case of the transform so we use a limiting argument.
\begin{lemma} \label{intozlem}  \, For $u>1$ and $s \in  (1/2,1)$ or $s=1/2+it$, where $t$ is a non\hyp{}zero real number, we have
    \begin{alignat*}{4}
        &\textup{(a)}& \quad  & \displaystyle J_s(u)\;&=& \; \; \displaystyle \gamma_J(s)G_{J}(u,s)u^{(3-s)/2}+\gamma_J(1-s)G_{J}(u,1-s)u^{1+s/2},&
 \\
&\textup{(b)}& \quad  &\displaystyle K_s(u)\;&=& -\gamma_K(s) G_{K}(u,s)u^{2-s}-\gamma_K(1-s) G_{K}(u,1-s)u^{s+1}-\frac{u\log{u} }{\pi s(1-s)}+C(s)u+O(u^{-1}t^{-4})&,
 \end{alignat*}
where
$C(s)$ is independent of $u$,
    \begin{flalign*}
\gamma_J(s):=&\frac{ \Gamma(1/2-s)}{4\Gamma(1-s/2)\Gamma(5/2-s/2)}, \\
\gamma_K(s):=&\frac{\left(\Gamma\left((1-s)/2\right)\right)^2\Gamma\left(1/2-s\right)}{16\left(\Gamma\left(1-s/2\right)\right)^2 \Gamma\left((3-s)/2\right) \Gamma\left(2-s/2\right)},\end{flalign*}
and 
\begin{flalign}
G_{J}(u,s):=&{}_2F_1\left( \frac{s+1}{2}, \frac{s-3}{2} \, ; s+\frac{1}{2} \, ; -u^{-1} \right), \label{gjdef} \\
G_K(u,s):=&
{}_3F_2\left(\frac{s}{2},\frac{s}{2}-1,\frac{s}{2}-\frac{1}{2};\frac{s}{2}+\frac{1}{2},s+\frac{1}{2};-u^{-2}\right). \label{gkdef} \end{flalign}

\end{lemma}
\begin{proof} \,
For (a), we apply equation (\ref{hgexp}) directly. \\
For (b), we first write $K_s(u)$ as 

        \begin{eqnarray}
K_s(u) = \frac{-u^3}{6\pi  }\lim_{\epsilon \rightarrow 0}{}_4F_3\left(1+\epsilon,1, \frac{s+1}{2}, \frac{2-s}{2} \, ; \frac{5}{2},2, \frac{3}{2} \, ; -u^2 \right) \nonumber.
\end{eqnarray}
Applying equation (\ref{hgexp}) to the perturbed hypergeometric function, we have
        \begin{eqnarray}
& K_s(u)& \nonumber \\
& \qquad =& -u^{2-s}\gamma_K(s) G_{K}(u,s)-u^{s+1}\gamma_K(1-s) G_{K}(u,1-s)+\frac{u }{2\pi s(1-s)}\lim_{\epsilon \rightarrow 0}\left(\frac{F(s,u,\epsilon)-1}{\epsilon}\right)
\nonumber \\
& \qquad =& -u^{2-s}\gamma_K(s) G_{K}(u,s)-u^{s+1}\gamma_K(1-s) G_{K}(u,1-s)+\frac{u }{2\pi s(1-s)}\frac{\partial F}{\partial \epsilon}(s,u,0),\nonumber \end{eqnarray}
where
\begin{equation} F(s,u,\epsilon):=\frac{u^{-2\epsilon} \cdot \Gamma\left(\frac{s-1}{2}-\epsilon\right)\Gamma\left(\frac{-s}{2}-\epsilon\right)\Gamma(\frac{1}{2})\Gamma(\frac{3}{2})}{\Gamma\left(\frac{s-1}{2}\right)\Gamma\left(\frac{-s}{2}\right)\Gamma(\frac{3}{2}-\epsilon)\Gamma(\frac{1}{2}-\epsilon)} {}_3 F_2 \left(
\begin{matrix}
\epsilon-\frac{1}{2},\epsilon, \epsilon+\frac{1}{2}
\\1+\epsilon+\frac{1-s}{2},1+\epsilon+\frac{s}{2} \end{matrix} 
;-u^{-2}
\right) . \nonumber 
\end{equation}
We now note that, by the series expansion of ${}_3F_2$ (see equation (\ref{hgseriesdef})), we have that
$$\left. \frac{\partial}{\partial \epsilon} 
{}_3 F_2 \left(
\begin{matrix}
\epsilon-\frac{1}{2},\epsilon, \epsilon+\frac{1}{2}
\\1+\epsilon+\frac{1-s}{2},1+\epsilon+\frac{s}{2} \end{matrix} 
;-u^{-2}
\right) 
\right \vert_{\epsilon=0} \ll u^{-2}t^{-2}. $$
Hence, we have
\begin{equation*}
    \frac{\partial F}{\partial \varepsilon}(s,u,0) = -2\log{u}-\psi\left(\frac{-s}{2}\right)-\psi\left(\frac{s-1}{2}\right)+\psi\left(\frac{3}{2}\right)+\psi\left(\frac{1}{2}\right)+O(u^{-2}t^{-2}),
\end{equation*}
where $\psi:=\Gamma'/\Gamma$ is the digamma function. Therefore,
\begin{equation} K_s(u)=-u^{2-s}\gamma_K(s) G_{K}(u,s)-u^{s+1}\gamma_K(1-s) G_{K}(u,1-s)-\frac{u\log{u} }{\pi s(1-s)}+C(s)u+O(u^{-1}t^{-4}),
\nonumber \end{equation}
where
$$C(s)=\frac{1}{2\pi s(1-s)} \left(-\psi\left(\frac{-s}{2}\right)-\psi\left(\frac{s-1}{2}\right)+\psi\left(\frac{3}{2}\right)+\psi\left(\frac{1}{2}\right) \right).$$

\end{proof}

Furthermore, using equation (\ref{hgexp}),  we prove the following lemma, establishing asymptotics for $K_s(R,r)$.
\begin{lemma} \label{ksrrlemma} For $r,R$ as in equation (\ref{rrdef}), $X>1$ and $s=1/2+it$, where $t$ is a non\hyp{}zero real number, we have
\begin{eqnarray}K_s(R,r) &=&\frac{1}{\pi s(1-s)}\left( \frac{R\log{R}-r\log{r}}{R-r}-\frac{\log{2}}{2}\right)  +O(|t|^{-7/2}\log{X}), \label{asy3}\end{eqnarray}
uniformly as $t,X \rightarrow \infty$.
\end{lemma}
\begin{proof}
Using equation (\ref{hgexp}) and equation (\ref{hypgo1}), we have
\begin{eqnarray} {}_3F_2\left(1, \frac{s+1}{2}, \frac{2-s}{2} \, ;\frac{7}{2}, \frac{3}{2} \, ; -2 \right) =\frac{5}{s(1-s)}\cdot \frac{1}{2}\cdot {}_3F_2\left(1, -\frac{3}{2}, \frac{1}{2} \, ;\frac{3-s}{2}, \frac{s+2}{2} \, ; -\frac{1}{2} \right)+O(|t|^{-7/2}). \nonumber\end{eqnarray}
By the series definition of the ${}_3F_2$ (see equation (\ref{hgseriesdef})), we deduce that
\begin{eqnarray} {}_3F_2\left(1, \frac{s+1}{2}, \frac{2-s}{2} \, ;\frac{7}{2}, \frac{3}{2} \, ; -2 \right)
&=&\frac{5}{2s(1-s)}+O(|t|^{-7/2}). \nonumber\end{eqnarray}
Hence,   \begin{eqnarray}K_s(R,r)&=& \frac{8a \sqrt{2}}{15\pi} \left(\frac{5}{2s(1-s)}+O(|t|^{-7/2}) \right) \nonumber \\ &=& \frac{1}{\pi s(1-s)}\left( \frac{R\log{R}-r\log{r}}{R-r}-\frac{\log{2}}{2}\right)  +O(|t|^{-7/2}\log{X}), \nonumber\end{eqnarray}
where $a$ is as in equation (\ref{aldef}).
\end{proof}
We note that, by Stirling's approximation formula, and for $s=1/2+it$, we have
\begin{eqnarray}
\label{gamma2}
\gamma_J(s) &\ll& |t|^{-5/2}, \\
\label{gamma3}
\gamma_K(s) &\ll& |t|^{-7/2}.
\end{eqnarray}

\subsection{Upper bounds for large \texorpdfstring{$X$}{X} and \texorpdfstring{$t$}{t} } \,

We now use the formulae of Lemma \ref{intozlem} together with equation (\ref{asy3}) to establish upper bounds for $G_J, \; G_K$ and their derivatives with respect to $u$. We then use them to obtain corresponding bounds for $d_t^{(1)}(f_3)$ and $d_t^{(1)}(f_4)$. We derive two sets of such bounds: one (see Lemma \ref{estimate1}) that is good enough to make the tail of the expansion given from the modified relative trace formula (Theorem \ref{modtrace}) absolutely convergent, and one (see Lemma \ref{estimate2}) that is better in the $X$ aspect with the expense of being worse in the $t$ aspect. We will apply the second set of bounds to the intermediate spectral terms. 
\begin{lemma} \label{gbound1} \, For $u >1 $ and $s=1/2+it$, where $t$ is a non\hyp{}zero real number or $t$ is a fixed number such that $s \in\left(1/2,1\right)$, we have 
\begin{flalign*}
&\textup{(a)}& &\textup{(i)} \quad \; G_J(u,s)=O(1),&&\textup{(ii)} \quad G_K(u,s)=O(1),& &&
\\
 &\textup{(b)}& &\textup{(i)} \quad G'_J(u,s)\ll |t| \cdot  u^{-2},&&\textup{(ii)} \quad G'_K(u,s) \ll |t|^{1/2} \cdot u^{-3},& &&
    \end{flalign*}
     uniformly as $u,t \rightarrow \infty$, where $G_J,\;G_K$ are as in equation (\ref{gjdef}) and equation (\ref{gkdef}).
\end{lemma}
\begin{proof}
For part (ii) of (a), by equation (\ref{hintrans}), we have 
\begin{equation}\label{intrep1}G_K(u,s)=\frac{1}{B\left(1/2,s/2 \right) }\int_{0}^{1}x^{s/2-1}(1-x)^{-1/2} {}_2F_1\left(\frac{s}{2}-\frac{1}{2},\frac{s}{2}-1;s+\frac{1}{2};-u^{-2}x\right)\, dx,\end{equation}
where $B(z_1,z_2)$ is the Beta function (see equation (\ref{betadef})). We demonstrate the case when $t$ is non\hyp{}zero real number. The case $s$ real is straightforward.

From  equation (\ref{quadtrans3}) and the series definition of the hypergeometric function, we have
$${}_2F_1\left(\frac{s}{2}-\frac{1}{2},\frac{s}{2}-1;s+\frac{1}{2};x\right)=\left( \frac{1-\sqrt{1-x}}{x}\right)^{s-1/2}(1-x)^{3/4}+O\left(|s|^{-1}\right).$$
Hence, setting \begin{equation} \label{adef} A(x)=\frac{\sqrt{1+u^{-2}x}-1}{\sqrt{x}}=\frac{u^{-2}\sqrt{x}}{1+\sqrt{1+u^{-2}x}}, \nonumber \end{equation} we get
\begin{eqnarray}G_K(u,s) &\ll& \sqrt{t} \left|\int_{0}^{1}x^{s/2-1}\left( \frac{1-\sqrt{1+u^{-2}x}}{-u^{-2}x}\right)^{s-1/2}(1+u^{-2}x)^{3/4} \cdot (1-x)^{-1/2}\, dx \right|+O\left(|s|^{-1/2}\right) \nonumber \\
&\ll& \sqrt{t} \left|\int_{0}^{1}A^{it}\left(\frac{1+u^{-2}x}{x}\right)^{3/4}\cdot (1-x)^{-1/2} \, dx\right|+O\left(|s|^{-1/2}\right) \nonumber \\
&\ll& \sqrt{t} \left|\int_{0}^{1-\epsilon}A^{it}\left(\frac{1+u^{-2}x}{x}\right)^{3/4}\cdot (1-x)^{-1/2} \, dx \right|+\sqrt{t} \int_{1-\epsilon}^{1} (1-x)^{-1/2} \, dx+O\left(|s|^{-1/2}\right) \nonumber \\
&\ll& \sqrt{t} \left|\int_{0}^{1-\epsilon}A^{it}\left(\frac{1+u^{-2}x}{x}\right)^{3/4}\cdot (1-x)^{-1/2} \, dx \right|+O\left(|t|^{1/2}\epsilon^{1/2}+|t|^{-1/2}\right), \nonumber
\end{eqnarray}
where $\epsilon\in(0,1)$ to be chosen later.

Note that $A'(x)=A(x)/\left(2x\sqrt{1+u^{-2}x}\right)$. Hence, we have \begin{flalign*} &\! \! \! G_K(u,s)  & \nonumber \\
 \ll& \; \sqrt{t} \left|\int_{0}^{1-\epsilon}A'\cdot A^{it-1}x^{1/4}\left(1+u^{-2}x\right)^{5/4}(1-x)^{-1/2} \, dx\right|+O\left(|t|^{1/2}\epsilon^{1/2}+|t|^{-1/2}\right)& \nonumber \\
\ll& \; \sqrt{t} \cdot\left| \frac{\epsilon^{-1/2}}{it}-\frac{1}{it}\int_{0}^{1-\epsilon}A^{it}\left(x^{1/4}\left(1+u^{-2}x\right)^{5/4}(1-x)^{-1/2}\right)'\, dx\right|+O\left(|t|^{1/2}\epsilon^{1/2}+|t|^{-1/2}\right)& \nonumber \\
 \ll& \; \frac{1}{t}\left|\int_{0}^{1-\epsilon}\left(x^{1/4}\left(1+u^{-2}x\right)^{5/4}(1-x)^{-1/2}\right)'\, dx\right|+O\left(|t|^{-1/2}\epsilon^{-1/2}+|t|^{1/2}\epsilon^{1/2}+|t|^{-1/2}\right)& \nonumber \\
\ \ll& \; |t|^{-1/2}\epsilon^{-1/2}+|t|^{1/2}\epsilon^{1/2}+|t|^{-1/2}. &\nonumber
\end{flalign*}
We optimize by taking $|t|^{-1/2}\epsilon^{-1/2}=|t|^{1/2}\epsilon^{1/2},$
i.e., $\epsilon=1/|t|$. This gives $$G_K(u,s) \ll 1.$$
For part (i) of (a), we apply equation (\ref{conti}) on $G_J(u,s)$ (defined in (\ref{gjdef})), and then we apply equation (\ref{quadtrans2}) combined with the series definition of the hypergeometric function.

For part (ii) of (b), we differentiate equation (\ref{intrep1}) to get
$$G'_K(u,s)=\frac{1}{B\left(\frac{1}{2},\frac{s}{2} \right) }\int_{0}^{1}x^{s/2-1}(1-x)^{-1/2}\cdot \left(\frac{2x}{u^3}\cdot \frac{\left(\frac{s}{2}-\frac{1}{2}\right)\left(\frac{s}{2}-1 \right)}{s+\frac{1}{2}}{}_2F_1\left(\frac{s}{2}+\frac{1}{2},\frac{s}{2};s+\frac{3}{2};-u^{-2}x\right) \right)\, dx.$$

We proceed similarly with part (ii) of (a) to get
$$G'_K(u,s) \ll  |t|^{1/2}u^{-3}.$$
For part (i) of (b), we have
        $$G'_{J}(u,s)=\frac{(s+1)(s-3)}{4s+2} \cdot u^{-2} \cdot {}_2F_1\left( \frac{s+3}{2}, \frac{s-1}{2} \, ; s+\frac{3}{2} \, ; \frac{-1}{u} \right).$$
        The result then follows from equation (\ref{quadtrans2}) combined with the series definition of the hypergeometric function.
\end{proof}
 Combining the estimates (\ref{gamma2}),
 (\ref{gamma3}) for $\gamma_J(s), \; \gamma_K(s)$
 with Lemma \ref{gbound1}(a), Lemma \ref{intozlem}, equations (\ref{mvt2}),(\ref{mvt3}) and equation (\ref{asy3}), we conclude the following lemma.

\begin{lemma}   \label{estimate1} \, For $X > 1$ and $s=1/2+it$, where $t$ 
 is a non\hyp{}zero real number, we have 
\begin{flalign*}
&\textup{(a)}&   & \displaystyle d_{t}^{(1)}(f_3) \ll X^{3/2}Y^{-1}|t|^{-5/2},&  &&\\
&\textup{(b)}&  & \displaystyle d_{t}^{(1)}(f_4) \ll X^{3/2}Y^{-1}|t|^{-7/2},& &&
\end{flalign*}
    uniformly as $X,|t| \rightarrow \infty$.
\end{lemma}
We now apply the mean value theorem on equations (\ref{mvt2}), (\ref{mvt3}) and use Lemma \ref{gbound1}(b) to find estimates that are stricter in the $X$ aspect but worse in the $t$ aspect.
\begin{lemma} \label{estimate2} \, For $X > 1$ and $t \neq 0$ real, we have 
\begin{flalign*}
&\textup{(a)}& &d_{t}^{(1)}(f_3) \ll X^{1/2}|t|^{-3/2},& && \\
&\textup{(b)}& &d_{t}^{(1)}(f_4) \ll X^{1/2}|t|^{-5/2},& &&
    \end{flalign*}
     uniformly as $X,|t| \rightarrow \infty$.
\end{lemma}
\begin{proof}
We demonstrate the proof of part (b). The proof of part (a) is similar.

   By equation (\ref{mvt3}), Lemma \ref{intozlem}(b) and the mean value theorem, we have that, for some $\xi \in \left[r,R\right] $,
    \begin{eqnarray}d_{t}^{(1)}(f_4)=&-&\gamma_K(s) \left(\xi^{2-s}G'_K(\xi,s)+(2-s)\xi^{1-s}G_K(\xi,s) \right) \nonumber \\
   &-&\gamma_K(1-s) \left(\xi^{s+1}G'_K(\xi,1-s)+(s+1)\xi^{S}G_K(\xi,1-s) \right) \nonumber \\
    &-& \frac{1}{\pi \lambda}\cdot \frac{R\log{R}-r\log{r}}{R-r}+C(s)+O(X^{-1}Y^{-1}t^{-4}) -\frac{1}{\sqrt{2}}K_s(\sqrt{2})+K_s(R,r). \nonumber 
    \end{eqnarray}
    Applying part (ii) of Lemma \ref{gbound1}(a) and equation (\ref{gamma3}) on Lemma \ref{intozlem}(b), we have
     $$\frac{1}{\sqrt{2}}K_s(\sqrt{2})=\frac{\log{2}}{2\pi s(1-s)}+C(s)+O(|t|^{-7/2}).$$
     Hence,
     \begin{eqnarray}d_{t}^{(1)}(f_4)=&-&\gamma_K(s) \left(\xi^{2-s}G'_K(\xi,s)+(2-s)\xi^{1-s}G_K(\xi,s) \right) \nonumber \\
   &-&\gamma_K(1-s) \left(\xi^{s+1}G'_K(\xi,1-s)+(s+1)\xi^{S}G_K(\xi,1-s) \right) \nonumber \\
    &-& \frac{1}{\pi \lambda}\cdot \left(\frac{R\log{R}-r\log{r}}{R-r}-\frac{\log{2}}{2}\right)+K_s(R,r)+O(|t|^{-7/2}), \nonumber
    \end{eqnarray}
    where the $C(s)$'s cancel out.
    We now observe that the last line of the equation above matches equation (\ref{asy3}). Therefore,
    \begin{eqnarray}d_{t}^{(1)}(f_4)=&-&\gamma_K(s) \left(\xi^{2-s}G'_K(\xi,s)+(2-s)\xi^{1-s}G_K(\xi,s) \right) \nonumber \\
   &-&\gamma_K(1-s) \left(\xi^{s+1}G'_K(\xi,1-s)+(s+1)\xi^{S}G_K(\xi,1-s) \right) \nonumber \\
    &+&O(|t|^{-7/2} \log{X}). \nonumber
    \end{eqnarray}
    Applying part (ii) of Lemma \ref{gbound1}(a) and part (ii) of Lemma \ref{gbound1}(b) together with equation (\ref{gamma3}),
 we conclude that
    \begin{eqnarray}d_{t}^{(1)}(f_4)\ll |t|^{-7/2}\left(\xi^{3/2}\cdot |t|^{1/2}\xi^{-3}+t\cdot \xi^{1/2}\right)+ 
|t|^{-7/2}\log{X} \ll |t|^{-5/2}\xi^{1/2} \ll |t|^{-5/2}X^{1/2}.  \nonumber
    \end{eqnarray}

    For part (a), we use equation (\ref{mvt2}) instead of (\ref{mvt3}), Lemma \ref{intozlem}(a) instead of \ref{intozlem}(b), part (i) of Lemma \ref{gbound1}(a),(b) instead of part (ii), and, finally, equation (\ref{gamma2}) instead of equation (\ref{gamma3}).
\end{proof}
For the case $t <\left(X/Y\right)^2= D^{-2},$ we provide more accurate estimates. While not needed for the proof of Theorem \ref{mainthm}, they will be a key ingredient in the proof of Theorem \ref{newmserr}. 
\begin{lemma} \label{sievecoeff}
For $X>1$, $X^{-1/2}<D=o(1)$, and $|t|< D^{-2}$, we have
\begin{flalign*}
&\textup{(a)}&   & \displaystyle d_{t}^{(1)}(f_3) = X^{1/2} (a(t,D) X^{it} + a(-t,D)  X^{-it})+O(|t|^{-5/2}),&  &&\\
&\textup{(b)}&  & \displaystyle d_{t}^{(1)}(f_4) = X^{1/2} (b(t,D) X^{it} + b(-t,D)  X^{-it})+O(|t|^{-7/2}\log{X}),& &&
\end{flalign*}
    where
   $$ a(t,D) \ll |t|^{-5/2} \cdot \mathrm{min} \left(D^{-1}, \; |t| \right),$$
   and
   $$b(t,D)  \ll |t|^{-7/2} \cdot \mathrm{min}\left(D^{-1}, \; |t|\right).$$
\end{lemma}
\begin{proof} We demonstrate the proof of part (b). The proof of part (a) is similar.
Via the series definition of the hypergeometric function and equation (\ref{gkdef}), we have, for $u>1$,
    $$ G_K(u,s)=1+O\left(t/u^2\right).$$
    Hence, via Lemma \ref{intozlem}, we have
    \begin{equation*}
K_s(u)=  \left(-\gamma_K(s)u^{2-s}-\gamma_K(1-s)u^{s+1}\right)\left(1+O(t/u^2) \right)-\frac{u\log{u} }{\pi s(1-s)}+C(s)u+O(u^{-1}|t|^{-4}).
\end{equation*}
   We combine this with the fact that $t < D^{-2} < X$, equation (\ref{mvt3}), and Lemma \ref{ksrrlemma}, to get
     \begin{eqnarray}d_{t}^{(1)}(f_4)=&-&\left(\gamma_K(s)\frac{R^{2-s}-r^{2-s}}{R-r}+\gamma_K(1-s)\frac{R^{1+s}-r^{1+s}}{R-r}\right)\left(1+O(X^{-1}) \right) \nonumber \\
   &+&O(X^{-1}|t|^{-4}+|t|^{-7/2}\log{X}).  \nonumber \end{eqnarray}
   We now use the fact that, for $z \in \mathbb{C}$, we have $$(u+\epsilon)^z-u^z \ll z \cdot \epsilon \cdot u^{z-1},$$
   to write
   $$R^{z}-(X+Y)^z \ll zX^{z-2}, \quad r^z-X^z \ll zX^{z-2},$$
   or, equivalently,
   $$R^{z}-X^{z}(D+1)^z\ll zX^{z-2}, \quad r^z-X^z \ll zX^{z-2}.$$
   Hence, using the fact that $t \ll D^{-2}$, we have

   \begin{eqnarray}d_t^{(1)}(f_4)=&-&\left(\gamma_K(s)\frac{(D+1)^{2-s}-1}{D}X^{1-s}+\gamma_K(1-s)\frac{(D+1)^{1+s}-1}{D}X^{s}\right)\left(1+O\left(X^{-1}\right) \right) \nonumber \\
   &+&O(|t|^{-7/2}\log{X}). \nonumber \end{eqnarray}
   We let 
   $$b(t,D):=-\gamma_K(1-s)\frac{(D+1)^{1+s}-1}{D}.$$
   As $D=o(1)$, the binomial theorem gives

  $$\frac{(D+1)^{1+s}-1}{D} \ll \hbox{min}\left(D^{-1},|t|\right).$$
   The bound 
   $$b(t,D)  \ll |t|^{-7/2} \cdot \text{min}\left(D^{-1}, \; |t|\right)$$
   now follows from equation (\ref{gamma3}).

   For the proof of part (a), we follow the same argument, using equation (\ref{mvt2}) instead of (\ref{mvt3}), and equation (\ref{gamma2}) instead of equation (\ref{gamma3}). The coefficients $a(t,D)$ are given by $$a(t,D):=\gamma_J(1-s)\frac{(D+1)^{2+s}-1}{(D+1)^2-1}.$$
\end{proof}

\subsection{Estimates corresponding to small eigenvalues} \,

We now provide estimates for the spectral coefficients corresponding to fixed, small eigenvalues, as $X\rightarrow \infty$. The case when $s_j \in \left(1/2,1\right)$ is real, will give the main terms (see Lemma \ref{lemmamterm}). We also show that the case of $s_j=1/2+it_j$ where $t_j \neq 0$ is a small real number contributes only an error of order $O(X^{1/2})$ (see Lemma \ref{medterms}). We consider the limiting case $t=0$ (i.e. $s_j=1/2$) independently in Lemma \ref{halflemma}.
\begin{lemma} \label{halflemma}
    For $t=0$, we have the following bounds:
    $$d^{(1)}_{0}(f_3), \; d^{(1)}_{0}(f_4)=O\left( X^{1/2}\log{X}\right).$$
\end{lemma}
\begin{proof}
    Applying the mean value theorem for the function $J_s(u)$ (see equation (\ref{mvtforhalf1})),
    we have that there  there is some $\xi \in [r^2,R^2]$ such that \begin{equation}\label{halfest0}d^{(1)}_{0}\left(f_3 \right)=\pi \int_{0}^{\sqrt{\xi}}\frac{x^2}{\sqrt{\xi-x^2}}\cdot {}_2F_1\left( \frac{3}{4}, \frac{3}{4} \, ; \frac{3}{2} \, ; -x^2 \right) dx.\end{equation}

    The integral representation of the hypergeometric function (equation (\ref{intrepr}) gives
\begin{eqnarray}{}_2F_1\left(\frac{3}{4},\frac{3}{4};\frac{3}{2};-x^2\right) &\ll& \int_{0}^{1}u^{-1/4}(1-u)^{-1/4}(1+x^2u)^{-3/4}\, du \nonumber \\ 
&\ll& \int_{0}^{x^{-2}}u^{-1/4}\, du+x^{-3/2}\int_{x^{-2}}^{1}u^{-1}\, du \nonumber \\
&\ll& \frac{\log{x}}{x^{3/2}}. \label{halfest}
\end{eqnarray}
We combine (\ref{halfest0}) and (\ref{halfest}) to conclude that
\begin{equation}d^{(1)}_{0}\left(f_3 \right)\ll  \int_{0}^{\sqrt{\xi}}\frac{x^2}{\sqrt{\xi-x^2}} \frac{\log{x}}{x^{3/2}} dx \ll \xi^{1/4}\log{\xi} \ll X^{1/2}\log{X}. \nonumber \end{equation}
For the case of $f_4$, we apply the mean value theorem for $K_s(u)$ (see equation (\ref{mvtforhalf2})). Then, we make use of equation (\ref{asy3}) and Lemma \ref{intozlem}(b) to show that equation (\ref{mvt3}) can be written as
$$d^{(1)}_{0}(f_4)=-\frac{\xi}{2\pi}\int_{0}^{\xi} \frac{x}{\sqrt{\xi^2-x^2}} \cdot {}_{3}F_{2}\left(1,\frac{3}{4},\frac{3}{4}; 2,\frac{3}{2};-x^2 \right) \, dx+O\left( \log{X}\right),$$
for some $\xi \in [r,R]$. \\
On the other hand, using equation (\ref{hintrans}), we have
$${}_{3}F_{2}\left(1,\frac{3}{4},\frac{3}{4}; 2,\frac{3}{2};-x^2 \right)=\int^{1}_{0}{}_{2}F_{1}\left(\frac{3}{4},\frac{3}{4}; \frac{3}{2} ; -x^2u \right) \, du.$$
Using equation (\ref{halfest}) this gives
$${}_{3}F_{2}\left(1,\frac{3}{4},\frac{3}{4}; 2,\frac{3}{2};-x^2 \right) \ll \frac{\log{x}}{x^{3/2}}.$$
Hence,
$$d^{(1)}_{0}(f_4) \ll \xi \int^{\xi}_{0} \frac{x\log{x}}{x^{3/2}\sqrt{\xi^2-x^2}} \, dx = \xi \int^{\xi}_{0} \frac{\log{x}}{\sqrt{x(\xi^2-x^2)}} \, dx \ll \xi \cdot \frac{\log{\xi}}{\xi^{1/2}}=\xi^{1/2}\log{\xi}. $$
It follows that $$d^{(1)}_{0}(f_4)=O(X^{1/2}\log{X}).$$
\end{proof}
We now proceed to provide asymptotics for the spectral coefficients corresponding to $s \in\left(1/2,1\right)$.
\begin{lemma} \label{lemmamterm} For fixed $s \in (1/2,1)$, we have
\begin{eqnarray}
&&d^{(1)}_t(f_3) = \frac{3-s}{2}\gamma_J(s)X^{1-s}+\frac{s+2}{2}\gamma_J(1-s)X^{s}+O(YX^{-s}+YX^{s-1}), \nonumber \\
&&d^{(1)}_{t}(f_4)=-(2-s)\gamma_K(s)X^{1-s}-(s+1)\gamma_K(1-s)X^s+O(X^{1/2}+YX^{-s}+YX^{s-1}). \nonumber
\end{eqnarray}

\end{lemma}
\begin{proof}
    We prove the first equation. The second is similar.
    
    We note that, by the series definition of the hypergeometric function, $$G_J(u,s) =1+O(u^{-1}).$$
Therefore, equation (\ref{mvt2}) and Lemma \ref{intozlem}(a), give

$$d^{(1)}_t(f_3)=\gamma_J(s)\frac{R^{3-s}-r^{3-s} }{R^2-r^2}+\gamma_J(1-s)\frac{R^{s+2}-r^{s+2} }{R^2-r^2}
+O\left(Y^{-1}X^{-s}+Y^{-1}X^{s-1} \right).$$
Using the mean value theorem we note that, for a fixed real number $w$, there is some $\xi \in \left[r^2,R^2\right]$ with $$\frac{R^{w}-r^{w}}{R^2-r^2}=\frac{w}{2}\xi^{w/2-1}= wX^{w-2}+O\left(wYX^{w-4} \right). $$
The result follows.
\end{proof}
The same argument also gives the following result.
\begin{lemma} \label{medterms} For fixed real $t \neq 0$,
    $$d_t^{(1)}(f_3),d_t^{(1)}(f_4) \ll X^{1/2}.$$
\end{lemma}

\section{Estimates of the Periods in Mean Square} \label{secperiods}
Huber\cite{huber} proved the following bound for the mean square of the periods $\hat{u}_{0,j}$.
\begin{lemma}[Huber, \cite{huber}]\label{lemmahuber} We have
$$\sum_{ t_j \leq T}{ \left| \hat{u}_{0,j} \right|^2}+\int_{-T}^{T} \left|\hat{E}_{\mathfrak{a},0}\left(\frac{1}{2}+it \right) \right|^2 dt \ll T.
$$
\end{lemma}
\begin{remark}
In fact, in \cite[Thm.1 p.2388]{tsuzuki}, Tsuzuki proved the following asymptotic:
\begin{equation*}
\sum_{ t_j \leq T}{ \left| \hat{u}_{0,j} \right|^2}+\int_{-T}^{T} \left|\hat{E}_{\mathfrak{a},0}\left(\frac{1}{2}+it \right) \right|^2 dt \sim \frac{\mathrm{len}(l)}{\pi} T.
\end{equation*}
\end{remark}
\begin{remark}
Lemma \ref{lemmahuber} is an important ingredient in the works of Huber \cite{huber}, Lekkas \cite{lekkas}, Chatzakos--Petridis \cite{chatzakos}, and Mckee \cite{mckee} .
\end{remark} \noindent
We will use a method similar to \cite{huber} to prove that the same estimate holds for $\hat{u}_{1,j}$ as well.

\begin{lemma} \label{periodslemma} We have
$$\sum_{ t_j \leq T}{ \left| \hat{u}_{1,j} \right|^2}+\int_{-T}^{T} \left|\hat{E}_{\mathfrak{a},1}\left(\frac{1}{2}+it \right) \right|^2 dt \ll T.$$
\end{lemma}

\begin{remark}
Notice that, by Cauchy--Schwarz inequality, we can deduce that
\begin{equation*}\sum_{ t_j \leq T} \hat{u}_{1,j}\hat{u}_{0,j} \ll \left(\sum_{t_j \leq T}{ \left| \hat{u}_{0,j} \right|^2}\right)^{1/2}\cdot \left(\sum_{t_j \leq T}{ \left| \hat{u}_{1,j} \right|^2}\right)^{1/2} \ll T. \end{equation*}
\end{remark}
\begin{remark} Once we have the estimates for the mean square of $\hat{u}_{0,j}$ and $\hat{u}_{1,j}$ as above, extending to higher weights is straightforward. Indeed, as demonstrated in \cite[Eq.1]{voskoulekkas}, it is easy to show the relations
\begin{align*}  \hat{u}_{m+2,j}&=-\sqrt{\frac{m^2+m+\lambda_j}{m^2+3m+2+\lambda_j}}\hat{u}_{m,j}, \\ \hat{E}_{\mathfrak{a}, m+2}(s)&=-\sqrt{\frac{m^2+m+s(1-s)}{m^2+3m+2+s(1-s)}} \hat{E}_{\mathfrak{a}, m}(s).\end{align*} 
Hence, we can deduce that, for every fixed $m$,
$$\sum_{ t_j \leq T}{ \left| \hat{u}_{m,j} \right|^2}+\int_{-T}^{T} \left|\hat{E}_{\mathfrak{a},m}\left(\frac{1}{2}+it \right) \right|^2 dt \ll T.
$$
\end{remark}
We define 
$$B_f(z):=\sum_{\gamma \in \Gamma_1 \backslash \Gamma}\left(j_{\gamma}(z)\right)^{-2} \cdot f\left( \frac{1}{ \cos^2\left( v\left( \gamma  z \right) \right)} \right),$$ 
where $j_{\gamma}(x)$ is as in equation (\ref{jfactor}), and $f$ is an appropriate test function.
Note that, for any $\gamma_0 \in \Gamma$, $$j_{\gamma}(\gamma_0 z)=j_{\gamma \gamma_0}(z) / j_{\gamma_0}(z).$$
Therefore,
\begin{eqnarray*}B_f(\gamma_0 z ) &=&
\sum_{\gamma \in \Gamma_1 \backslash \Gamma}\left(j_{\gamma}(\gamma_0 z)\right)^{-2} \cdot f\left( \frac{1}{ \cos^2{\left( v\left( \gamma \gamma_0  z \right) \right)}}  \right)\nonumber \\
&=& j^2_{\gamma_0}(z) \cdot B_f(z). \nonumber\end{eqnarray*}
We deduce that $B_f$ is an element of $\mathfrak{h}_1$. 

 In a similar manner with Lemma \ref{spside}, we consider the spectral expansion of $B_f(z)$. For simplicity, we assume that $\Gamma$ is cocompact.
 \begin{lemma} \label{bfexp}
     For $f$ a continuous, piecewise differentiable function with exponential decay, we have
$$B_f(z)=\sum_{j} c_{j}(f)u_{1,j}(z),$$
where \begin{equation} \label{cif} c_j(f)=\hat{u}_{1,j}d_{1,t_j}(f)-i\hat{u}_{0,j}d_{2,t_j}(f), \end{equation}
and the real transforms $d_{1,t_j}(f)$ and $ d_{2,t_j}(f)$ are defined via \begin{alignat*}{2}
    &d_{1,t_j}(f) & {}:={}& 2\int_{0}^{\pi/2} \frac{1}{\cos{v}} f\left( \frac{1}{ \cos^2{v}} \right) \Phi^{e}_j(v) \, dv, \\
    &d_{2,t_j}(f) & {}:={}& 2\sqrt{\lambda_j}\int_{0}^{\pi/2} \frac{\sin{v}}{\cos^2{v}} f\left( \frac{1}{ \cos^2{v}} \right) \Phi^{o}_j(v) \, dv.
\end{alignat*}
Here, $\Phi^{o,e}_j(v)$ are, respectively, the odd and even solutions of the differential equation
$$\cos^2{v} \left( \frac{d^2V_j}{d v^2} +V_j \right)+\lambda_j V_j=0 $$
that satisfy the initial conditions
$$\Phi^o_j(0)=0, \; \left(\Phi^o_j\right)'(0)=1, \quad \hbox{and} \quad \Phi^e_j(0)=1, \; \left(\Phi^e_j\right)'(0)=0.$$
 \end{lemma}
 \begin{proof}
  By the Spectral Theorem for $\mathfrak{h}_1$ (see \cite{roelcke}), we have that, for some coefficients $c_{j}(f)$,
  $$B_f(z)=\sum_{j} c_{j}(f)u_{1,j}(z).$$
Fixing a particular $u_{1,i}$, to calculate $c_i(f)$, we multiply both sides by $\bar{u}_{1,i}(z)$ and we proceed to integrate over the fundamental region of $\Gamma$. This gives
  \begin{align*}c_i(f) &=  \int_{-\pi/2}^{\pi/2} \frac{1}{\cos^2{v}} f\left( \frac{1}{ \cos^2\left( v \right)} \right)\bar{U}_i(v) \, dv  \\  &= \int_{0}^{\pi/2} \frac{1}{\cos^2{v}} f\left( \frac{1}{ \cos^2\left( v \right)} \right)\left(\bar{U}_i(v)+\bar{U}_i(-v) \right) \, dv,  \end{align*}
where, for $z=ie^{u+iv}$, we have $$U_i(v):= \int_{0}^{\mathrm{len}(l)} u_{1,i}(z) \, du. $$
 Note that $$\hat{u}_{1,i}=U_i(0),$$
 and that, from \cite[Eq.(3),(8)]{fay}, we have
 $$u_{0,j}(z)=\frac{-i}{\sqrt{\lambda_j}}\overline{K_{-1}}u_{1,j}(z).$$ 
 Hence,
 using equation (\ref{rop}), we have\begin{align*}\hat{u}_{0,i}=  \int_{0}^{\mathrm{len}(l)}\left. u_{0,i}\right\vert_{v=0} \, du &=\frac{-i}{\sqrt{\lambda_i}}\int_{0}^{\mathrm{len}(l)} \left. \overline{K_{-1}}u_{1,i}\right\vert_{v=0} \, du \\  &= \frac{-i}{\sqrt{\lambda_i}}\int_{0}^{\mathrm{len}(l)} \left. \left( e^{iv}\cos{v}  \left(\frac{\partial}{\partial u}+i\frac{\partial}{\partial v} \right)-1 \right)u_{1,i} \right\vert_{v=0} du  \\
 &=\frac{U'_i(0)+iU_i(0)}{\sqrt{\lambda_i}}.
 \end{align*}

 Recall now that $u_{1,i}(z)$ satisfies the equation $\left(D_1+\lambda_i \right)u_{1,i}(z)=0$ for some eigenvalue $\lambda_i$, where $D_1$ is the Laplacian in $\mathfrak{h}_1$ (see equation (\ref{dms})).
 By direct computation, it is easy to see that, in Huber coordinates, $D_1$ takes the form
 \begin{equation*}  D_1=\cos^2{v} \left ( \frac{\partial ^2}{\partial u^2} + \frac{\partial^2}{\partial v^2} \right)+2i\cos^2{v}\frac{\partial}{\partial v} +i\sin{2v}\frac{\partial}{\partial u}. \end{equation*}
 Hence, we have that
 $$\cos^2{v} \left ( \frac{\partial ^2 u_{1,i}}{\partial u^2} + \frac{\partial^2u_{1,i}}{\partial v^2} \right)+2i\cos^2{v}\frac{\partial u_{1,i}}{\partial v} +i\sin{2v}\frac{\partial u_{1,i}}{\partial u}+\lambda_i u_{1,i}=0.$$
 Integrating for $u$ in $\left[0,\mathrm{len}(l)\right]$ and using the periodicity of $u_{1,i}$ in the $u$ variable, we get 
 $$\cos^2{v} \left( U_i''+2iU'\right)+\lambda_iU_i=0. $$
 Now write $U_i= e^{-iv}V_i$.
 Note that $$U'_i=e^{-iv}\left(V'_i-iV_i \right), \quad \hbox{and} \quad U''_i=e^{-iv}\left(V''_i-2iV'_i-V_i \right). $$
 Therefore, the differential equation becomes $$\cos^2{v} \left( V''_i +V_i \right)+\lambda_iV_i=0. $$
 Clearly, if $V_i(v)$ is a solution of the equation, then so is $V_i(-v)$. We set \begin{equation}\label{phiedef} \Phi^{e}_i(v)=\frac{1}{2V_i(0)} \left(V_i(v)+V_i(-v) \right), \quad \Phi^{o}_i(v)=\frac{1}{2V'_i(0)} \left(V_i(v)-V_i(-v) \right),  \end{equation}
 so that $$V_i(v)=V_i(0)\Phi^{e}_i(v)+V'_i(0)\Phi^{o}_i(v),$$ and $\Phi^{o,e}_i(v)$ satisfy the initial conditions $$\Phi^o_i(0)=0, \; \left(\Phi^o_i\right)'(0)=1, \quad \hbox{and} \quad \Phi^e_i(0)=1, \; \left(\Phi^e_i\right)'(0)=0.$$
 By uniqueness, the functions $\Phi^{o,e}_i$ are both real. By definition, we can see that $$V_i(0)=U_i(0)=\hat{u}_{1,i}, \quad 
 \hbox{and} \quad V'_i(0)=U'_i(0)+iU_i(0)=\sqrt{\lambda_i} \hat{u}_{0,i}.$$
 Combining the above, we have $$U_i(v)+U_i(-v)=2\hat{u}_{1,i}\cos{v}\cdot\Phi^{e}_i(v)-2i\hat{u}_{0,i}\sqrt{\lambda_i}\sin{v} \cdot\Phi^{o}_i(v).$$
Hence,
$$c_i(f)=\hat{u}_{1,i}d_{1,t_i}(f)-i\hat{u}_{0,i}d_{2,t_i}(f),$$
where \begin{alignat*}{2}
    &d_{1,t_i}(f) & {}:={}& 2\int_{0}^{\pi/2} \frac{1}{\cos{v}} f\left( \frac{1}{ \cos^2{v}} \right) \Phi^{e}_i(v) \, dv, \\
    &d_{2,t_i}(f) & {}:={}& 2\sqrt{\lambda_j}\int_{0}^{\pi/2} \frac{\sin{v}}{\cos^2{v}} f\left( \frac{1}{ \cos^2{v}} \right) \Phi^{o}_i(v) \, dv.
\end{alignat*}
\end{proof}

In particular, as both the periods $\hat{u}_{0,i},\hat{u}_{1,i}$ and the transforms $d_{1,t_i}(f)$, $d_{2,t_i}(f)$ are real, via equation (\ref{cif}), we have that
$$ \left| c_i(f) \right|^2=\left| \hat{u}_{1,i}d_{1,t_i}(f) \right|^2+\left| \hat{u}_{0,i}d_{2,t_i}(f) \right|^2.$$
Recall now Perseval's identity:
\begin{equation} \label{parsevalid} \sum_{i}{\left|c_i(f)\right|^2}=\int_{ \Gamma \backslash \mathbb{H}} |B_f(z)|^2 d \mu(z). \end{equation}
Below, we will use this, for an appropriate choice of $f$, to find an upper bound for the periods. In particular, for a fixed $v_0 \in [0,\pi/2]$, we choose $f=f_0$ to be the indicator function of  $[1,\sec^2{v_0}]$ (as in \cite[Eq.(78)]{huber}).
\subsection{The Proof of Lemma \ref{periodslemma}}
We proceed with a similar method as \cite[Appendix 4.1,4.2]
{huber}.

Recall that the function $\Phi_i^{e}$ defined in equation (\ref{phiedef}) is a solution of the equation
$$\cos^2{v} \cdot F''+\left(\lambda_i+\cos^2{v}\right)F=0, \; \; F(0)=1, \; F'(0)=0.$$
We multiply by $F'$ to get the following:

$$\frac{\cos^2{v}}{\lambda_i+\cos^2{v}}\frac{d}{dv}\left(F'(v)\right)^2=-\frac{d}{dv}\left(F(v)\right)^2.$$
\\
Now we integrate in some subinterval of $[0,\pi/2]$, say $[0,v']$. We have
\begin{eqnarray}-F^2\left( v'\right)+F^2(0)&=&\int_{0}^{v'}\frac{\cos^2{v}}{\lambda_i+\cos^2{v}}\frac{d}{dv}\left(F'(v)\right)^2 dv \nonumber \\ &=&\frac{\cos^2{v'}}{\lambda_i+\cos^2{v'}}\left(F'(v')\right)^2+\int_{0}^{v'}\frac{\lambda_i \sin{2v}}{\left(\lambda_i+\cos^2{v}\right)^2}\left(F'(v)\right)^2 dv \geq 0  \nonumber \end{eqnarray}
Hence, for every $v$ in $[0,\pi/2]$, we have
$1 \geq F^2(v)$. Therefore, $$  F''(v)=-\frac{\lambda_i+\cos^2v}{\cos^2{v}}F(v) \geq -\frac{\lambda_i+\cos^2v}{\cos^2{v}}. $$
We integrate twice to get
 \begin{equation}F(v) \geq 1+\lambda_i \log{\left(\cos{v} \right)}-\frac{1}{2}v^2 \geq 1-\frac{1}{2}\left(\lambda_i \tan^2{v}+v^2\right). \nonumber \end{equation}
 Hence, we have
 \begin{eqnarray}
 d_{1,t_i}(f_0)
 &=&2\int_{0}^{v_0} \frac{1}{\cos{v}}  \Phi_i^{e}(v) \, dv \nonumber \\  
  &\geq& 2\cos{v_0}\int_{0}^{v_0} \frac{1}{\cos^2{v}}  \Phi_i^{e}(v) \, dv \nonumber \\
 &\geq& 2\cos{v_0}\int_{0}^{v_0} \frac{1-\left(\lambda_i\tan^2{v}+v^2\right)/2}{\cos^2{v}}  \, dv\nonumber \\
 &=& \cos{v_0}\left(2\tan{v_0}-\frac{\lambda_i}{3}\tan^3{v_0} -\int_{0}^{v_0} \frac{v^2}{\cos^2{v}}  \, dv \right) \nonumber \\
 & \geq &  \cos{v_0}\left(\left(2-v_0^2\right)\tan{v_0}-\frac{\lambda_i}{3}\tan^3{v_0} \right). \nonumber
 \end{eqnarray}
If we take $\lambda_i \leq 2/\tan^2{v_0}=:X$ we can easily deduce that, for small $v_0$,

$$d_{1,t_i}(f_0)   \geq \tan{v_0}=\sqrt{2/X}.$$
Therefore, it follows that
$$\sum_{i}{\left|c_i(f_0)\right|^2} \geq \sum_{\lambda_i \leq X}{\left|c_i(f_0)\right|^2} \geq \sum_{\lambda_i \leq X}{ \left| \hat{u}_{1,i}d_{1,t_i}(f_0) \right|^2}  \geq \frac{2}{X} \sum_{\lambda_i \leq X}{ \left| \hat{u}_{1,i} \right|^2} .$$

Note now that $|B_{f_0}(z)| \leq |A^{(0)}_{f_0}(z)|$, so by \cite[Eq. 61]{huber} we have

$$\int_{\mathbb{H}/\Gamma} |B_{f_0}(z)|^2 d \mu(z) \leq \int_{\mathbb{H}/\Gamma} |A^{(0)}_{f_0}(z)|^2 d \mu(z) \ll X^{-1/2}. $$

We conclude the proof of Lemma \ref{periodslemma} for the cocompact case by combining the last two inequalities with equation (\ref{parsevalid}). For the general cofinite case, we also need the following estimates for the Eisenstein periods.
\begin{lemma} \label{eisensteinperiods}
   As $T \rightarrow +\infty$, we have the following bounds:
   \begin{align}
    \int_{-T}^{T} \left|\hat{E}_{\mathfrak{a},0}\left(\frac{1}{2}+it \right) \right|^2 dt \ll T, \label{eperiodbound1} \\
    \int_{-T}^{T} \left|\hat{E}_{\mathfrak{a},1}\left(\frac{1}{2}+it \right) \right|^2 dt \ll T. \label{eperiodbound2}
   \end{align}
\end{lemma}
\begin{proof}
    For the first part, see \cite[Lemma 4.3]{chatzakos}. The second part follows similarly, using the series $B_f(z)$ instead of $A^{(0)}_f(z)$ in the proof from \cite{chatzakos}.
\end{proof}

\section{Proof of Theorem \ref{equivmain}}
\label{mainthmsection}
We demonstrate the case of $N_4(X,l)$. The case of $N_3(X,l)$ is similar.

From Theorem \ref{modtrace} and equation (\ref{countingsub}), we have
$$N_4(X,l)=2\sum_{j} \lambda_jd^{(1)}_{t_j}(f_4)\hat{u}^2_{1,j}+O\left(Y+X^{2/3}\right).$$
Using Lemma \ref{halflemma} to deal with the case $t_j=0$, we have
\begin{equation}\label{twosums} \sum_{j} \lambda_jd^{(1)}_{t_j}(f_4)\hat{u}^2_{1,j}=\sum_{\frac{1}{2} < s_j < 1}\lambda_jd^{(1)}_{t_j}(f_4)\hat{u}^2_{1,j}+\sum_{t_j\in \mathbb{R}-\left\{0\right\}}\lambda_jd^{(1)}_{t_j}(f_4)\hat{u}^2_{1,j}+O\left(X^{1/2}\log{X}\right).\end{equation}
Noting that the first sum in the right hand side of equation (\ref{twosums}) is finite, we apply Lemma \ref{lemmamterm} to get
\begin{equation}\sum_{\frac{1}{2} < s_j < 1}\lambda_jd^{(1)}_{t_j}(f_4)\hat{u}^2_{1,j}
=-\sum_{\frac{1}{2} < s_j < 1}\lambda_j(s_j+1)\gamma_K(1-s_j)\hat{u}^2_{1,j}X^{s_j}+O\left(X^{1/2}+Y\right).\nonumber
\end{equation}
For the second sum  in the right hand side of equation (\ref{twosums}), we use Lemma \ref{medterms} to deal with the terms corresponding to small $t_j$'s.
\begin{eqnarray}\sum_{t_j\in \mathbb{R}-\left\{0\right\}}\lambda_jd^{(1)}_{t_j}(f_4)\hat{u}^2_{1,j}&=&\sum_{\substack{t_j\in \mathbb{R}-\left\{0\right\}:\\ |t_j| \geq 1}}\lambda_jd^{(1)}_{t_j}(f_4)\hat{u}^2_{1,j}+ \sum_{\substack{t_j\in \mathbb{R}-\left\{0\right\}:\\ |t_j| < 1}}\lambda_jd^{(1)}_{t_j}(f_4)\hat{u}^2_{1,j} \nonumber \\
&=& \sum_{ |t_j| \geq 1}\lambda_jd^{(1)}_{t_j}(f_4)\hat{u}^2_{1,j} + O\left(X^{1/2}\right).\nonumber
\end{eqnarray}
We now use a dyadic decomposition for the bulk of the spectrum.
$$\sum_{ |t_j| \geq 1}\lambda_jd^{(1)}_{t_j}(f_4)\hat{u}^2_{1,j} \ll \sum_{ t_j \geq 1}\lambda_j\left|d^{(1)}_{t_j}(f_4)\right|\hat{u}^2_{1,j} = \sum_{n=0}^{\infty}\left( \sum_{2^n \leq t_j <2^{n+1}}\lambda_j\left|d^{(1)}_{t_j}(f_4)\right|\hat{u}^2_{1,j} \right).$$
Using Lemma \ref{estimate1}(b) and Lemma \ref{estimate2}(b), this becomes 
\begin{eqnarray}\sum_{ |t_j| \geq 1}\lambda_jd^{(1)}_{t_j}(f_4)\hat{u}^2_{1,j} &\ll& \sum_{n=0}^{\infty}\left( 2^{2n} \cdot \hbox{min}\left(2^{-7n/2}X^{3/2}Y^{-1},2^{-5n/2}X^{1/2} \right)\sum_{2^n \leq t_j <2^{n+1}}\hat{u}^2_{1,j} \right) \nonumber \\
&\ll& \sum_{n=0}^{\infty}\left(\hbox{min}\left(2^{-3n/2}X^{3/2}Y^{-1},2^{-n/2}X^{1/2} \right)\sum_{2^n \leq t_j <2^{n+1}}\hat{u}^2_{1,j} \right).
\nonumber\end{eqnarray}
Finally, we apply Lemma \ref{periodslemma}.
\begin{eqnarray}\sum_{ |t_j| \geq 1}\lambda_jd^{(1)}_{t_j}(f_4)\hat{u}^2_{1,j}
&\ll& X^{1/2}\sum_{n=0}^{\infty}\left( \hbox{min}\left(2^{-3n/2}XY^{-1},2^{-n/2} \right)\cdot 2^{n}\right) \nonumber \\
&\ll& X^{1/2}\sum_{n=0}^{\infty}\left( \hbox{min}\left(2^{-n/2}XY^{-1},2^{n/2} \right)\right)\nonumber \\
&=&X^{1/2}\left( \sum_{2^n \leq XY^{-1}}2^{n/2}+XY^{-1}\cdot \sum_{2^n > XY^{-1}}2^{-n/2} \right) \nonumber \\
&\ll& X^{1/2}\left(X^{1/2}Y^{-1/2}+XY^{-1}\cdot\left(XY^{-1}\right)^{-1/2} \right)\ll XY^{-1/2}. \nonumber \end{eqnarray}
Therefore,
$$N_4(X,l)=-2\sum_{\frac{1}{2} < s_j < 1}\lambda_j(s_j+1)\gamma_K(1-s_j)\hat{u}^2_{1,j}X^{s_j}+O\left(Y+XY^{-1/2} +X^{2/3} \right). $$
This is optimized for $Y=XY^{-1/2}$, i.e., $Y=X^{2/3}$.
This concludes the proof of the theorem for $\Gamma$ cocompact. For $\Gamma$ a general cofinite group, to find asymptotic formulae for $N_i(X,l)$, we also need to take into account the contribution of the continuous spectrum in the spectral side of the modified relative trace formula (Theorem \ref{modtrace}). We demonstrate the cases where $i=1$ and $i=4$. The cases $i=2$ and $i=3$ are similar.
It is enough to show that, for $f_1,f_4$ as in section \ref{choicesection},
\begin{equation*} E^{(a)}(f_1), \; E^{(c)}(f_4) \ll XY^{-1/2}+X^{1/2}. \end{equation*}
The bound $E^{(a)}(f_1) \ll XY^{-1/2}+X^{1/2}$ follows from (\ref{eperiodbound1}) and the work of Lekkas in \cite{lekkas}. The bound $E^{(c)}(f_4) \ll XY^{-1/2}+X^{1/2}$ follows from (\ref{eperiodbound2}) and our estimates from section \ref{estsection} in a similar manner as the cocompact case. To deal with the case of $t$ being close to $0$, we notice that all the instances of $t^{-A}$ in the estimates can be straightforwardly replaced with $(1+|t|)^{-A}$. The cases $i=2$ and $i=3$ are similar.

\section{Estimates of the Error in Mean Square} \label{sectionmse}
To prove Theorem \ref{newmserr}, we use the large sieve inequalities provided in \cite{voskoulekkas}. Let $a_j$ be a sequence of complex numbers, and let $a_{\mathfrak{a}}(t)$ be a sequence of continuous complex functions indexed by the cusps $\mathfrak{a}$. Fix a non-negative integer $m$ to be the weight.
\begin{theorem} \label{sievevar}
    Let $T,X>1$ and $\delta>0$. Let $x_1, \dots , x_R \in [X,2X]$. If $|x_\nu-x_\mu| > \delta >0$ for $\nu \neq \mu$, then
    \begin{equation*}\sum_{\nu=1}^R \Big| \sum_{|t_j|\le T}a_jx_\nu^{it_j}\widehat{u}_{m,j}+\frac{1}{4\pi}\sum_{\mathfrak{a}}\int_{-T}^{T}a_{\mathfrak{a}}x_{\nu}^{it}\hat{E}_{\mathfrak{a}, m}(1/2+it) \, dt\Big|^2 \ll \big(T+X\delta^{-1}\big)\|\mathbf{a}\|_*^2, \end{equation*}
 where $$\|\mathbf{a}\|_* := \left(\sum_{|t_j| \le T}|a_j|^2+\frac{1}{4\pi}\sum_{\mathfrak{a}}\int_{-T}^{T}|a_{\mathfrak{a}}(t)|^2 \, dt \right)^{1/2} \;.$$
\end{theorem}
Similarly to equation (\ref{errordef}), for $i=1,\dots,4$, we define $$E_i(X):=N_i(X)-M_i(X),$$
where $$M_i(X):=\frac{2\delta_{1i}\left(\mathrm{len}(l)\right)^2}{\pi\mathrm{Vol}\left(\Gamma \backslash \mathbb{H}\right)}X+\frac{1}{2 \sqrt{\pi}}\sum_{1/2<s_j < 1}c_i(s_j)X^{s_j},$$
and $N_i(X)$ is as in Theorem \ref{equivmain}.
We use Theorem \ref{sievevar} to prove the following result.
\begin{theorem} Let $X>1$ and let $X_1, \dots , X_R \in [X,2X]$. Let $\delta>0$ be such that $|X_\nu-X_\mu| > \delta$ for every $\nu \neq \mu$.
For $i=1,\dots 4$, we have
\begin{equation*}\sum_{n=1}^R|E_{i}(X_n,l)|^2 
\ll \mathrm{min}\left(X^{3/2},  X^{4/3}R^{1/3}\log{X}\right) +X^2\delta^{-1} \log^2{X}
,\end{equation*}
and, hence, for $\mu,\mu' \in \left\{1,-1\right\}$, by equation (\ref{ndeltaequiv}),
\begin{equation*}\sum_{n=1}^R|E^{\mu,\mu'}(X_n,l)|^2 
\ll \mathrm{min}\left(X^{3/2},  X^{4/3}R^{1/3}\log{X}\right) +X^2\delta^{-1} \log^2{X}
.\end{equation*}
\end{theorem}
\begin{proof}
For simplicity, we assume that $\Gamma$ is cocompact. The general cofinite case follows similarly, as the same arguments can be used to deal with the contribution of the continuous part of the spectrum. We refer to \cite[Prop. 5.3] {chatzakos}. For the case $i=1$, see \cite{lekkas}.
The case $i=2$ is equivalent to $i=3$ (see Remark \ref{th1rem1}). Therefore, we are left to consider $i=3$ and $i=4$. We demonstrate the case $i=4$. The case $i=3$ is similar, with part (b) of Theorem \ref{modtrace} being used in place of part (c).

Let $D \in [0,1]$ be such that $Y > X^{1/2} \log{X}$. Recall that $Y:=DX$. Let $f_4$  be as in equation (\ref{f4def}).
Let
$$\tilde{N}_{4}(X):=I'_{f_4,1}(0),$$
where $I_{f,1}$ is as in Proposition \ref{geoside}, and
$$\tilde{M}_{4}(X):=2\sum_{j:\lambda_j<1/4} \lambda_jd^{(1)}_{t_j}(f_4)\hat{u}^2_{1,j}.$$
Finally, let $$\tilde{E}_4(X):=\tilde{N}_4(X)-\tilde{M}_4(X).$$
We have
\begin{equation}\tilde{E}_4(X)=E_{4}(X)+O\left( |\tilde{N}_4(X)-N_4(X)|+|\tilde{M}_4(X)-M_4(X)|\right). \nonumber \end{equation}
By Lemma \ref{lemmamterm}, we have that
$$|\tilde{M}_4(X)-M_4(X)| \ll Y.$$
Furthermore, we have
$$|\tilde{N}_4(X)-N_4(X)| \ll |N_1(X+Y)-N_1(X)| \ll |E_1(X)|+Y.$$
We conclude that
\begin{equation}\label{eitoefi}E_4(X) \ll |\tilde{E}_{4}(X)|+|E_1(X)|+Y.\end{equation}
We now split the spectrum in the following intervals:

$$A_1=\{t_j : 0<|t_j|\le 1\}, \quad A_2=\{t_j : 1\le |t_j| \le D^{-2}\}, \quad A_3=\{t_j : |t_j|>D^{-2}\} \; .$$

Let also 

$$ 
S^{(4)}_n:=2\sum_{t_j\in A_n}\lambda_jd^{(1)}_{t_j}(f_4)\widehat{u}^2_{1,j}\; ,$$
so that,
$$ \tilde{E}_{4}(X)=S^{(4)}_1+S^{(4)}_2+S^{(4)}_3 \; .$$
Via Lemma \ref{estimate2}, we have
$$S^{(4)}_1 \ll X^{1/2}. $$
On the other hand, via Lemmata \ref{estimate1} and \ref{estimate2} we have
$$S^{(4)}_3 \ll X^{1/2}D^{-1}\sum_{|t_j|>D^{-2}}|t_j|^{-3/2}|\widehat{u}_{1,j}|^2.$$

By partial summation and Lemma \ref{periodslemma}, we conclude that
$$S^{(4)}_3 \ll X^{1/2}D^{-1} \cdot D=X^{1/2} < Y.$$
Hence, using a dyadic decomposition, we get
$$ |\tilde{E}_{4}\left(X\right)| \ll \left|S_2^{(4)}\right|+Y \ll \sum_{1\le T=2^k \le D^{-2}}\left|S^{(4)}(X,T)\right|+Y,$$
where $$S^{(4)}(X,T):= \sum_{T<|t_j|\le 2T}2\lambda_jd^{(1)}_{t_j}(f_4)\widehat{u}^2_{1,j}. \;$$ 
Adding over distinct $X_1,X_2,\dots X_R \in [X,2X]$, we get
$$ \sum_{n=1}^R|\tilde{E}_{4}(X_n)|^2 \ll \sum_{n=1}^R \Big(\sum_{1\le T=2^k \le 4D^{-2}}\left|S^{(4)}(X_n,T)\right|\Big)^2+RY^2.$$
Via the Cauchy--Schwarz inequality, we deduce that
\begin{eqnarray}\sum_{n=1}^R|\tilde{E}_4(X_n)|^2 &\ll& \log{X} \sum_{n=1}^R \sum_{1\le T=2^k \le 4D^{-2}}\Big| S^{(4)}(X_n,T)\Big|^2+RY^2 \; \nonumber \\
&\ll&
\log{X}
\sum_{1\le T=2^k \le 4D^{-2}}\sum_{n=1}^R \Big| S^{(4)}(X_n,T)\Big|^2+RY^2
. \nonumber \end{eqnarray}
We now deal with the remaining terms.
Via Lemma \ref{sievecoeff} b), we have
\begin{eqnarray}S^{(4)}(X_n,T) &\ll& \left| \sum_{T<|t_j|\le 2T}2\lambda_j\left(X_n^{1/2} (b(t_j,D) X_n^{it_j} + b(-t_j,D)  X_n^{-it_j})+O(t_j^{-7/2}\log{X})\right)\widehat{u}^2_{1,j} \right| \; , \nonumber \\
&\ll& X^{1/2}\left|\sum_{T<|t_j|\le 2T}\lambda_jb(t_j,D) X_n^{it} \widehat{u}^2_{1,j}\right|+ T^{-3/2}\log{X}\sum_{T<|t_j|\le 2T}\widehat{u}^2_{1,j}. \label{applyperiods}
\end{eqnarray}
Applying Lemma \ref{periodslemma} for the second sum of (\ref{applyperiods}), we have
\begin{eqnarray}S^{(4)}(X_n,T) \ll X^{1/2}\left|\sum_{T<|t_j|\le 2T}\lambda_jb(t_j,D) X_n^{it} \widehat{u}^2_{1,j}\right|+ T^{-1/2}\log{X}. \nonumber 
\end{eqnarray}
Hence,
$$\sum_{n=1}^R \Big| S^{(4)}(X_n,T)\Big|^2\ll X \sum_{n=1}^{R} \left| \sum_{T<|t_j|\le 2T}\lambda_jb(t_j,D) X_n^{it} \widehat{u}^2_{1,j}\right|^2+RT^{-1}(\log{X})^2.$$
We apply Theorem \ref{sievevar} with $m=1$ and $a_j=\lambda_jb(t_j,D)\widehat{u}_{1,j}$ when $T<|t_j| \leq 2T$ and $a_j=0$ otherwise, to get
$$\sum_{n=1}^R \Big| S^{(4)}(X_n,T)\Big|^2\ll X(T+X\delta^{-1})\|\mathbf{a}\|_*^2+RT^{-1}(\log{X})^2.$$
Furthermore, via Lemma \ref{sievecoeff}, we have
$$\|\mathbf{a}\|_*^2 = \sum_{|t_j|<2T}|a_j|^2 \ll \sum_{T<t_j<2T} t_j^{-3} \hbox{min}\left(D^{-2},t^2_j\right) \widehat{u}^2_{1,j} \ll T^{-2}\hbox{min}\left(D^{-2},T^2\right),$$
and, hence,
$$\sum_{n=1}^R \Big| S^{(4)}(X_n,T)\Big|^2\ll X(T+X\delta^{-1})T^{-2}\min\{T^2,D^{-2}\}+RT^{-1}(\log{X})^2. $$
Therefore,
\begin{eqnarray}\sum_{n=1}^R|\tilde{E}_{4}(X_n,l)|^2 &\ll& X\log{X} \sum_{1\le T=2^k \le 4D^{-2}}(T+X\delta^{-1})T^{-2}\min\{T^2,D^{-2}\}+RY^2 \; \nonumber \\
&\ll& X\log{X}\left(
D^{-1}+X\delta^{-1}\log{X}+D^{-2}\left(D+X\delta^{-1}D^{2} \right)\right)+RY^2 \nonumber \\
&\ll& X\log{X}\left(
D^{-1}+X\delta^{-1}\log{X}\right)+RY^2 \nonumber \\
&\ll& X^2Y^{-1} \log{X} +X^2\delta^{-1} \log^2{X}+RY^2 \nonumber
.\end{eqnarray}
We will now make an appropriate choice for $Y$.
If $X^{2/3}R^{-1/3}>X^{1/2}\log{X}$, take $Y=X^{2/3}R^{-1/3}$, so that $$X^2Y^{-1}=RY^2=X^{4/3}R^{1/3}<X^{3/2}/\log{X}.$$Otherwise, to ensure that $Y>X^{1/2}\log{X}$, take $Y=2X^{1/2}\log{X}$. We have
\begin{equation*}\sum_{n=1}^R|\tilde{E}_{4}(X_n,l)|^2 
\ll \hbox{min}\left(X^{3/2},  X^{4/3}R^{1/3}\log{X}\right) +X^2\delta^{-1} \log^2{X}
.\end{equation*}
By equation \ref{eitoefi} and \cite[Prop 1.4]{lekkas}, we deduce 
\begin{equation*}\sum_{n=1}^R|E_{4}(X_n,l)|^2 
\ll \hbox{min}\left(X^{3/2},  X^{4/3}R^{1/3}\log{X}\right) +X^2\delta^{-1} \log^2{X}
.\end{equation*}
The second part of the theorem follows from the first via equation (\ref{ndeltaequiv}).
\end{proof}
 
Choosing $\delta \gg X/R$ and $R>X^{1/2}$,
we have
\begin{equation*}\sum_{n=1}^R|E^{\mu,\mu'}(X_n)|^2 \frac{X}{R}
\ll\frac{X^3\delta^{-1}\log^2{X}}{R} \ll X^2\log^2{X} 
.\end{equation*}
Taking $X_n$ to be equidistanced, we have
\begin{equation*}\lim_{R \rightarrow +\infty}\sum_{n=1}^R|E^{\mu,\mu'}(X_n)|^2 \frac{X}{R}
=\int_{X}^{2X}\left|E^{\mu,\mu'}(x) \right|^2 \, dx
.\end{equation*}
Hence,
\begin{equation*}\frac{1}{X}\int_{X}^{2X}\left|E^{\mu,\mu'}(x) \right|^2 \, dx \ll X\log^2{X}
,\end{equation*}
concluding the proof of Theorem \ref{newmserr}.
\section{Arithmetic Applications} \label{arithmeticsection}
In this section, we prove Theorem \ref{ideals} about correlation sums involving $\mathcal{N}(n)$, the number of ideals of $\mathbb{Z}\left[\sqrt{2}\right]$ with norm $n$.
We achieve this by applying Theorems \ref{mainthm} and \ref{newmserr} for certain choices of $\Gamma$ related to quaternion orders. 

 In particular, consider the quaternion algebra $\displaystyle \left(\frac{q,r}{\mathbb{Q}}\right)$ generated by $i,j$ with $i^2=q$, $j^2=r$, and $ij=-ji$, where $q$ prime and $r$ either prime or equal to $\pm 1$.
This algebra embeds into $M_{2\times 2}\left(\mathbb{R} \right)$
via the map
$$\Phi: \, u+vi+sj+tij \longmapsto \begin{pmatrix}
u+v\sqrt{q} & s+t\sqrt{q}\\
r\left(s-t\sqrt{q}\right) & u-v\sqrt{q}
\end{pmatrix}, \quad u,v,s,t \in \mathbb{Q}.$$
For the case $r=p$, a fixed prime, and $q=2$, let $\mathcal{J}$ be the order in $\displaystyle 
 \left(\frac{2,p}{\mathbb{Q}}\right)$ with $\mathbb{Z}$-basis $1,i,j,ij$. Let $\Gamma$ be the group defined by \begin{equation} \Gamma :=\Phi\left(\mathcal{J}\right)\cap \hbox{SL}_2\left(\mathbb{R}\right)/\left\{\pm I \right\}= \left\{ \begin{pmatrix}
a & b\\
c & d
\end{pmatrix}  \left\vert \begin{array}{l} a=u+v\sqrt{2}, \; b=s+t\sqrt{2}, \, \\ c=p\left(s-t\sqrt{2}\right), \; d=u-v\sqrt{2}, \\  u,v,s,t \in \mathbb{Z}, \; ad-bc=1 \end{array} \right. \right\}\slash\left\{ \pm I \right\}, \label{quatgroup} \end{equation}
i.e., the intersection of $\Phi(\mathcal{J})$ and $\hbox{SL}_2\left(\mathbb{R}\right)$, modulo $\pm I$. By \cite[Prop 3.2]{hejhalquat}, $\Gamma$ is a cofinite Fuchsian group.
Let
$$\gamma_1=\gamma_2=h=\begin{pmatrix} \epsilon^2 & 0 \\ 0 & \epsilon^{-2} \end{pmatrix},$$ where $\epsilon = 1+\sqrt{2}$ is the fundamental unit of $\mathcal{O}_{\mathbb{Q}\left(\sqrt{2}\right)}=\mathbb{Z}\left[\sqrt{2}\right]$. Note that $$h^n \begin{pmatrix}
a & b\\
c & d
\end{pmatrix} h^{m}=\begin{pmatrix}
\epsilon^{2(n+m)} a & \epsilon^{2(n-m)}b\\
\epsilon^{-2(n-m)}c & \epsilon^{-2(n+m)}d\end{pmatrix}, $$
so, for $\Gamma_1=\Gamma_2=\left\langle h\right\rangle$,  the class of $\begin{pmatrix}
a & b\\
c & d
\end{pmatrix}$ in $\Gamma_1 \backslash \Gamma \slash \Gamma_2$ is $$ \left\{ \left. \begin{pmatrix}
\epsilon^{2\alpha} a & \epsilon^{2\beta}b\\
\epsilon^{-2\beta}c & \epsilon^{-2\alpha}d\end{pmatrix} \right\vert \alpha,\beta \in \mathbb{Z} \, , \alpha = \beta \
 (\hbox{mod}\ 2)  \right\}. $$
This implies a four-to-one map between  $\Gamma_1 \backslash \Gamma \slash \Gamma_2-\left\{\hbox{id}\right\}$ and the set of non-trivial ideals $\mathfrak{a}$, $\mathfrak{b}$ of $\mathbb{Z}\left[\sqrt{2}\right]$ with $\left|N(\mathfrak{a})-pN(\mathfrak{b}) \right|=1$, defined by  $$\begin{pmatrix}
a & b\\
c & d
\end{pmatrix} \longmapsto \left((a),(b) \right).$$
Indeed, the pre-image of every pair of such ideals $\mathfrak{a}$, $\mathfrak{b}$  consists of the pairwise distinct classes of $$\begin{pmatrix}
 \pm a & b\\
c & \pm d\end{pmatrix} , 
\begin{pmatrix}
 \pm \epsilon^2a & b\\
c & \pm \epsilon^{-2}d\end{pmatrix},$$
where $a,b$ are generators satisfying $\left|N(\mathfrak{a})-pN(\mathfrak{b}) \right|=N(a)-pN(b)$. Such generators always exist, as $\mathbb{Z} \left[ \sqrt{2} \right]$ has narrow class number $1$. 
Note that, for $p>2$, there does not exist an element $\gamma'=\begin{pmatrix}
a' & b'\\
c' & d'
\end{pmatrix}$ with $a'=\epsilon \cdot a$, $d'=-\epsilon^{-1} \cdot d$, otherwise $-ad=a'\cdot d' \equiv 1 \equiv ad \pmod{p}$. For $p=2$, this would give that $N(b')$ and $N(b)$ have different parities, and, therefore, the pairs $((a),(b))$ and $ \, ((a'),(b'))$ cannot be equal.

Recall 
$$N(X,l)=\#\left\{ \gamma \in \Gamma_1 \backslash \Gamma \slash \Gamma_1 \vert  |ad+bc| \leq X \right\}.$$
Note that, as $ad$ and $bc$ are both non-integers and $bc$ is a multiple of $p$, the equation $ad-bc=1$ gives that either $ad$ and $bc$ have the same sign or one of them is equal to $0$. Hence, in the notation described above, we have $$|ad+bc|=N(\mathfrak{a})+pN(\mathfrak{b}).$$
Therefore, 
\begin{align*}
N(X,l)&=4\#\left\{\mathfrak{a}, \mathfrak{b}  \vert  \left|N(\mathfrak{a})-pN(\mathfrak{b}) \right|=1 \; , \; N(\mathfrak{a})+pN(\mathfrak{b}) \leq X \right\}+1  \\
&=4\#\left\{\mathfrak{a}, \mathfrak{b}  \vert  \left|N(\mathfrak{a})-pN(\mathfrak{b}) \right|=1 \; , \; N(\mathfrak{b})  \leq \frac{X}{2p}\right\}+O\left(X^{\epsilon} \right).  
\end{align*}
Similarly,
\begin{equation*}
N^{+1,+1}(X,l)+N^{-1,-1}(X,l)=4\#\left\{\mathfrak{a}, \mathfrak{b}  \vert  N(\mathfrak{a})-pN(\mathfrak{b})=1 , \; N(\mathfrak{b})  \leq \frac{X}{2p} \right\}+O\left(X^{\epsilon} \right),  
\end{equation*}
and
\begin{equation*}
N^{+1,-1}(X,l)+N^{-1,+1}(X,l)=4\#\left\{\mathfrak{a}, \mathfrak{b}  \vert  N(\mathfrak{a})-pN(\mathfrak{b})=-1, \; N(\mathfrak{b})  \leq \frac{X}{2p} \right\}+O\left(X^{\epsilon} \right).  
\end{equation*}
Therefore, via Theorem \ref{mainthm}, we have
$$\sum_{n \leq X}\mathcal{N}(n)\mathcal{N}(pn\pm 1)
=4c_p^{-1}p \cdot \left(\frac{\log{\epsilon}}{\pi}\right)^2 X+\sum_{1/2 \leq s_j < 1}a^{\pm}_jX^{s_j}+O\left(X^{2/3}\right),$$
where $c_p=\hbox{Vol}\left(\Gamma \backslash \mathbb{H}\right)/2\pi$.
 We will compute $c_p$ explicitly by identifying $\Gamma$ as the conjugate of a finite-index subgroup of a group that corresponds to a maximal quaternion order.
\begin{lemma}
\label{cplemma}
    We have that
\begin{equation*}\textup{Vol}\left(\Gamma \backslash \mathbb{H}\right) =\left\{
\begin{array}{ll}
      \displaystyle 2(p-1)\pi, & p=  \pm 3\,  \left( \mathrm{mod} \; 8 \right), \\
      \displaystyle 2 (p+1)\pi, & p=  \pm 1 \,  \left( \mathrm{mod} \; 8 \right),\\
      \displaystyle 2p\pi, & p=2, \\
\end{array} 
\right.
\end{equation*}
and, hence,
\begin{equation*}c_p =p+\left(\frac{2}{p} \right) =\left\{
\begin{array}{ll}
      \displaystyle p-1, & p=  \pm 3\,  \left( \mathrm{mod} \; 8 \right), \\
      \displaystyle p+1, & p=  \pm 1 \,  \left( \mathrm{mod} \; 8 \right),\\
      \displaystyle p, & p=2. \\
\end{array} 
\right.
\end{equation*}
\end{lemma}
\begin{proof}
For $p= 5 \pmod{8}$, consider the quaternion algebra with $q=2$ and $r=p$, and the Eichler order $\mathcal{O}_B(2p,1)$, with $\mathbb{Z}$-basis $1,i,(1+j)/2,(i+ij)/2$. Consider further the group $\Gamma(2p,1)$, defined by
$$\Gamma(2p,1):=\Phi\left(\mathcal{O}_B(2p,1)\right)\cap \hbox{SL}_2\left(\mathbb{R}\right)/\left\{\pm I \right\}.$$

By \cite[Prop 2.29]{quaternions}, we have that $\Gamma(2p,1)$ is a cocompact Fuchsian group with volume equal to $(p-1)\pi/3$. We can check manually that $\Gamma$ is a subgroup of $\Gamma(2p,1)$ with index $6$, and, therefore, of volume $\hbox{Vol}\left(\Gamma \backslash\mathbb{H}\right)=2(p-1)\pi$. Indeed, the $6$ corresponding cosets can be identified as follows:  We have two cosets defined by the relations $(u,s,t)=(1,0,0),\, (1,0,1) \! \pmod{2}$, two defined by $(u,v,s)=(0,0,1),\, (0,1,1) \!  \pmod{2}$, one defined by $(u,v,s,t)=(1,v,1,v) \! \pmod{2}$, and, finally, one defined by $(u,v,s,t)=(1,v,1,v+1)\! \pmod{2}$. Hence, in this case, $c_p=p-1$.

For $p = 3 \pmod{8}$, consider again the quaternion algebra with $q=2$ and $r=p$. Consider further the order $\mathcal{O}$ with $\mathbb{Z}$-basis $1,i,(1+i+j)/2,(i+ij)/2$. The discriminant of this order is $2p$, implying that it is maximal. By \cite[Prop 2.29]{quaternions}, the cocompact Fuchsian group $\Gamma'$ defined by $$\Gamma':=\Phi\left(\mathcal{O}\right)\cap \hbox{SL}_2\left(\mathbb{R}\right)/\left\{\pm I \right\}$$ has finite volume, equal to $(p-1)\pi/3$. In a similar manner with the case $p=5 \pmod{8}$, we can show that $\Gamma$ is a subgroup of $\Gamma'$ with index $6$, giving, as before, $c_p=p-1$.

For $p=2$ or $p= \pm 1 \pmod{8}$, the corresponding quaternion algebra is a matrix algebra, and, therefore, the problem is reduced to congruence groups.  Indeed, using the fact that $p$ is a norm of an element of $\mathbb{Z}\left[\sqrt{2}\right]$, i.e., that the equation $p=x^2-2y^2$ has a solution, we can see that $\Gamma$ is conjugate to
\begin{equation*}
\Gamma'' := \left\{ \begin{pmatrix}
a & 2b\\
c & d
\end{pmatrix}  \left\vert \, \begin{array}{l} 2 \lambda\vert(a-d)+(b-c) \sqrt{2}\\  a,b,c,d \in \mathbb{Z}, \; ad-2bc=1 \end{array} \right. \right\}/\left\{ \pm I \right\}. \end{equation*}

For the case $p=2$, 
 
it follows that 
$\Gamma''$ has $\Gamma(4)$ as a subgroup of index $4$, with the four cosets being determined by the parity of $c/2$ and the parity of $(a+1)/2$. On the other hand, the index of $\Gamma(4)$ in $\pslz$ is $4^3\cdot(1-1/2^2)=48$ (see \cite[p.44]{iwaniec}). Therefore, $\Gamma''$ is a  congruence group of level $4$ and index $48/4=12$. This gives $\hbox{Vol}\left(\Gamma \backslash\mathbb{H}\right)=4 \pi$, and hence, $c_2=2=p$.

For the case $p= \pm 1 \pmod{8}$,
note that $\Gamma''$ is a congruence group of level dividing $4p$. 
Reducing modulo $4p$, we find that the index of  $\Gamma(4p)$ in $\Gamma''$ is $8p(p-1)$. On the other hand, the index of $\Gamma(4p)$ in $\pslz$ is $48p(p^2-1)$ (see \cite[p.44]{iwaniec}). Therefore, the index of $\Gamma''$ in $\pslz$ is $48p(p^2-1)/(8p(p-1))=6(p+1)$. This gives $\hbox{Vol}\left(\Gamma \backslash\mathbb{H}\right)=2 \pi (p+1)$, and hence $c_p=p+1$.
\end{proof}
This concludes the proof of the first part of Theorem \ref{ideals}. The second part now follows from Theorem \ref{newmserr}. 

We note that, due to Selberg's $1/4$ eigenvalue conjecture and the explicit Jacquet–Langlands
correspondence, we expect the middle sum in equation (\ref{corras}) to be empty. In other words, we have the following proposition.
\begin{proposition}
    Under Selberg's $1/4$ eigenvalue conjecture, for any fixed prime number $p$, we have
    \begin{equation} \sum_{n \leq X}\mathcal{N}(n)\mathcal{N}(pn\pm 1)=\frac{4p}{c_p} \cdot    \left(\frac{\log{\epsilon}}{\pi}\right)^2X+O(X^{2/3}), \end{equation}
    where $c_p$ is as in Theorem \ref{ideals}.
\end{proposition}
For more details, as well as a list of values of $p$ for which this is known unconditionally, see Remark \ref{jaclan}.

\section*{Acknowledgements}

The author is grateful to his PhD advisor, Yiannis Petridis, for suggesting the problem and for his advice and patience.

The author was supported by University College London and the Engineering and Physical Sciences Research Council (EPSRC) studentship grant EP/V520263/1. The author was also supported by the Swedish Research Council under grant no. 2016-06596 while in residence at Institut Mittag-Leffler in Djursholm, Sweden, during the Analytic Number Theory program in the spring of 2024. The author was further supported by the MTA--HUN-REN RI Lend\"{u}let Automorphic Research Group. Finally, the author thanks the Max Planck Institute for Mathematics in Bonn for its hospitality and financial support.

\section*{Statements and Declarations}
The author has no conflicts of interest, financial or non-financial, to
declare which are relevant to the content of this article.
 \appendix
\section{Special Functions } \label{spapp}
For $p,q$ non-negative integers with $p>q$, $z$ complex with $|z|<1$ and $a_i,b_i$ real numbers with $b_i$ not being non-positive integers, we define the hypergeometric function
by the power series
\begin{equation} \label{hgseriesdef} {}_pF_q\left( \begin{matrix}
a_1,\dots,a_p
\\b_1, \dots, b_q \end{matrix} 
; z
\right) = \sum_{n=0}^{\infty} \frac{\prod_i (a_i)_n}{\prod_i (b_i)_n} \cdot \frac{z^n}{n!},\end{equation}
where $(x)_n=x \cdot (x+1) \dots (x+n-1)$ and $(x)_0=1$.
This has an analytic continuation in any region avoiding the branch cut point $z=1$. 

For $p=1,q=0$, we have
$${}_1F_0\left(a; \, ; z \right)=(1-z)^{-a}.$$
For $p=2,q=1$, we have that ${}_2F_1\left(a,b; c \, ; z \right)$ is a solution of the differential equation
$$zF''(z)+\left(c-(a+b+1)z\right)F'(z)-abF(z)=0,$$ with initial condition $F(0)=1, \; F'(0)=ab/c$.

For any non-negative integers $p,q$ with $p>q$ and complex numbers $t,r,u$ with $|u|<1$ and $\hbox{Re}(t),\hbox{Re}(r)>0$, we have (see \cite[Eq.7.512.12, p.814]{toisap}) that
\begin{equation} \label{hintrans}
    \int_{0}^{1}(1-x)^{t-1}x^{r-1}\cdot {}_p F_q \left(
\begin{matrix}
a_1,\dots,a_p
\\b_1, \dots, b_q \end{matrix} 
; ux
\right) \, dx = B(t,r) \cdot {}_{p+1} F_{q+1} \left(
\begin{matrix}r,
a_1,\dots,a_p
\\t+r,b_1, \dots, b_q \end{matrix} 
; u
\right),
\end{equation}
where \begin{equation} \label{betadef} B(t,r):=\frac{\Gamma(t) \Gamma(r)}{\Gamma(t+r)}\end{equation} is the Beta function.
In particular, for $p=1,q=0$ we have
\begin{equation} \label{intrepr}
B(t,r)\cdot{}_{2}F_{1}\left( r,a\, ; \, t+r \, ; \, u \right)=\int_{0}^{1}(1-x)^{t-1}x^{r-1}(1-ux)^{-a} \, dx. \end{equation}
By analytic continuation, these can be extended to $u$ in a fixed region that avoids the branch point $u=1$ (for our purposes, we take this region to be defined by $\hbox{Re}(u) \leq 0$) .
Equations (\ref{hintrans}) and (\ref{intrepr}) can also be verified using the
series definition directly. In the same manner, we can prove the following:
\begin{equation} \label{hder} \frac{d}{dx}\left( {}_p F_q \left(
\begin{matrix}
a_1,\dots,a_p
\\b_1, \dots, b_q \end{matrix} 
; x
\right)\right)=\frac{\prod a_i}{\prod b_i} \cdot {}_p F_q \left(
\begin{matrix}
a_1+1,\dots,a_p+1
\\b_1+1, \dots, b_q+1 \end{matrix} 
; x
\right).\end{equation}

We also note the following contiguous relation, which can be proved using the series definition combined with analytic continuation:

\begin{equation} \label{conti} {}_2F_1\left(a,b ; \, c; \, z \right)={}_2F_1\left(a+1,b ; \, c; \, z \right)-\frac{bz}{c}{}_2F_1\left(a+1,b+1 ; \, c+1; \, z \right). \end{equation}
For any fixed real number $z$ and complex number $r$ bounded away from the negative real axis with $r \rightarrow \infty$ (see \cite[Eq.11, p.237]{luke}):
\begin{equation} \label{hypgoas}
{}_2 F_1 \left(
\begin{matrix}
r,r+c \,
\\2r+b \end{matrix} 
; z
\right)=
    \sqrt{\pi}\frac{ r^{-1/2}\Gamma(2r+b)}{\Gamma(r+c) \Gamma(r+b-c)} \cdot \frac{\left(\sqrt{1-z} \right)^{b-c-1/2}} {(1+\sqrt{1-z})^{2r+1-b}} \cdot \left(1+ O(r^{-1})\right).
\end{equation}
In particular, for $\hbox{Re}(r)$ bounded, we have
\begin{equation} \label{hypgo1}
    {}_{2}F_{1}\left( r,r+c\, ; \, 2r+b \, ; \, z \right) \ll \frac{ r^{-1/2}\Gamma(2r+b)}{\Gamma(r+c) \Gamma(r+b-c)} \ll 1.
\end{equation}
If $a_j-a_i$ is not an integer  and $a_j-b_i$ is not a non-negative integer for any pair of distinct $i,j$, then, by \cite[16.8.8]{dlmf} we have
\begin{equation} \label{hgexp} {}_{q+1} F_q \left(
\begin{matrix}
a_1,\dots,a_{q+1}
\\b_1, \dots, b_q \end{matrix} 
; x
\right)=\sum_{j=1}^{q+1}\gamma_j \cdot w_j(x),
\end{equation}
where
\begin{equation*}
    w_j(x):=(-x)^{-a_j} \cdot {}_{q+1} F_q \left(
\begin{matrix}
a_j, 1-b_1+a_j, \dots, 1-b_q+a_j 
\\1-a_1+a_j, \dots \, * \, \dots, 1-a_{q+1}+a_j \end{matrix} 
; \frac{1}{x}
\right),
\end{equation*}
and
\begin{equation*} \displaystyle
    \gamma_j:=\prod_{k \neq j}\frac{\Gamma(a_k-a_j)  }{ \Gamma(a_k)} \cdot \prod_{k}\frac{\Gamma(b_k)  }{ \Gamma(b_k-a_j)}.
\end{equation*}
The symbol `$*$' indicates that the entry $1-a_j+a_j$ is omitted.
If the necessary conditions are not satisfied and, say, $a_2-a_1$ is an integer, we consider instead the case $\tilde{a}_2=a_2+\epsilon$ and take the limit as $\epsilon \rightarrow 0$.

We now provide three quadratic transformations used in our work. The following pair of transformations (see \cite[15.8.20]{dlmf} and \cite[15.8.19]{dlmf}) is used in establishing estimates: For $\alpha,\beta$ non-zero complex numbers, we have
\begin{align} 
{}_2F_1\left(\alpha,1-\alpha ; \, \beta; \, z \right)&=(1-z)^{\beta-1}{}_2F_1\left(\frac{\beta-\alpha}{2},\frac{\beta+\alpha-1}{2}; \beta; 4z(1-z) \right),\label{quadtrans2}
\\ 
{}_2F_1\left(\alpha,1-\alpha ; \, \beta; \, z \right)&=(1-2z)^{1-\alpha-\beta}(1-z)^{\beta-1}{}_2F_1\left(\frac{\beta+\alpha}{2},\frac{\beta+\alpha-1}{2}; \beta; \frac{4z(1-z)}{(1-2z)^2} \right),\label{quadtrans3}
\end{align} where $\hbox{Re}(z)<1/2$ and $\beta$ is not a non-positive integer. 

Finally, we note the following quadratic transformation (see \cite[Eq.9.136.3, p.1009]{toisap}):
For $\alpha,\beta$ non-zero complex numbers, we have
\begin{eqnarray} 
    {}_2F_1\left(\alpha,\beta; \, \frac{\alpha+\beta+1}{2}; \, \frac{1+\sqrt{z}}{2} \right)- \,{}_2F_1\left(\alpha,\beta; \, \frac{\alpha+\beta+1}{2}; \, \frac{1-\sqrt{z}}{2} \right) \nonumber \\=
\frac{\Gamma((\alpha+\beta+1)/2)}{\Gamma(\alpha/2)\Gamma(\beta/2)} \cdot \sqrt{\pi z} \cdot {}_2F_1\left(\frac{\alpha+1}{2},\frac{\beta+1}{2}; \, \frac{3}{2}; \, z \right), \label{quadtrans}
\end{eqnarray}
where we require $\alpha/2$, $\beta/2$, and $(\alpha+\beta+1)/2$ not to be non-positive integers.


\begin{thebibliography}{99}
\bibitem{quaternions} M. Alsina and P. Bayer. \emph{Quaternion orders, quadratic forms, and Shimura curves.} CRM Monograph Series, 22. American Mathematical Society, Providence, RI, 2004. xvi+196 pp.
\bibitem{booker} A.R. Booker, M. Lee, A. Str{\"o}mbergsson. \emph{Twist-minimal trace formulas and the Selberg eigenvalue conjecture.} Journal of the London Mathematical Society, 2020.
\bibitem{chamizo1} F. Chamizo. \emph{The large sieve in riemann surfaces.} Acta Arithmetica, 77, no.4, 303–
313, 1996.
\bibitem{chamizo2} F. Chamizo. \emph{Some applications of large sieve in Riemann surfaces.} Acta Arithmetica, 77, no.4, 315–337, 1996.
\bibitem{chatzakos}
D. Chatzakos and Y. Petridis. 
\emph{The hyperbolic lattice point problem in conjugacy classes.} Forum Math., 28:981–1003, 2016.
\bibitem{petru} P. Constantinescu.
\emph{Dissipation of correlations of holomorphic cusp forms}, arXiv:2112.01427, 2021.
\bibitem{good} A. Good.
\emph{Local analysis of Selberg's trace formula. Lecture Notes in Mathematics, 1040.} Springer-Verlag, Berlin, 1983. i+128 pp.

\bibitem{toisap} I. S. Gradshteyn and I. M. Ryzhik. 
\emph{Table of integrals, series, and products. Translated from the Russian.} Translation edited and with a preface by Alan Jeffrey and Daniel Zwillinger. Seventh edition. Elsevier/Academic Press, Amsterdam, 2007. xlviii+1171 pp.

\bibitem{gunther}P. G\"{u}nther. \emph{Gitterpunktprobleme in symmetrischen Riemannschen R\"aumen vom Rang 1.} Math. Nachr., 94:5–27, 1980.

\bibitem{fay} J. D. Fay.
\emph{Fourier coefficients of the resolvent for a Fuchsian group.} J. Reine Angew. Math. 293/294, 143--203, 1977.
\bibitem{hejhal} D. A. Hejhal.
\emph{Sur certaines s{\'e}ries de Dirichlet dont les p{\^o}les sont sur les lignes critiques.} CR Acad. Sci. Paris S{\'e}r. A. 287, 383--385, p.2, 1978
\bibitem{hejhal1} D. A. Hejhal. \emph{Sur certaines séries de Dirichlet associées aux
géodésiques fermées d'une surface de Riemann compacte.}
 CR Acad. Sci. Paris S{\'e}r. I. 294, p.273--276, 1982.
\bibitem{hejhal2} D. A. Hejhal. \emph{Sur quelques propriétés asymptotiques des périodes
hyperboliques et des invariants algébriques d'un sous-groupe discret de PSL(2, R).}
 CR Acad. Sci. Paris S{\'e}r. I. 294, p.509--512, 1982.
\bibitem{hejhal3} D. A. Hejhal. \emph{Quelques exemples de séries de Dirichlet dont les pôles ont
un rapport étroit avec les valeurs propres de l'opérateur de Laplace--Beltrami hyperbolique.}
 CR Acad. Sci. Paris S{\'e}r. I. 294, p.637--640, 1982.
\bibitem{hejhalquat} D.A. Hejhal. \emph{A classical approach to a well-known spectral correspondence on quaternion groups.} In: Chudnovsky, D.V., Chudnovsky, G.V., Cohn, H., Nathanson, M.B. (eds) Number Theory. Lecture Notes in Mathematics, vol 1135. Springer, Berlin, Heidelberg, 1985.
\bibitem{huber} H. Huber.
\emph{Ein Gitterpunktproblem in der hyperbolischen Ebene.} J. Reine Angew. Math. 496, 15--53, 1998.


\bibitem{iwaniec} H. Iwaniec. \emph{Spectral Methods of Automorphic Forms}, 2nd ed., Graduate Studies in Mathematics, 53, American Mathematical Society, Providence, RI; Revista Matemática Iberoamericana, Madrid, xii+220 pp., 2002.




\bibitem{katok}
S. Katok. \emph{Fuchsian Groups}, University of Chicago Press, 1992

\bibitem{lekkas} D. Lekkas.
\emph{A relative trace formula and counting
geodesic segments in the hyperbolic
plane.}
Doctoral thesis (Ph.D), UCL (University College London), 2023.

\bibitem{voskoulekkas} D. Lekkas, M. Voskou. \emph{Large sieve inequalities for periods of Maass forms}, Res. Number Theory, 10, no. 3, Paper No. 73, 19 pp., 2024.

\bibitem{luke} Y. L. Luke. 
\emph{The special functions and their approximations}, volume 1. Academic
Press, 1969.

\bibitem{mckee} K. Martin, M. Mckee, and E. Wambach. 
\emph{A relative trace formula for a compact
Riemann surface.} International Journal of Number Theory, 07(02):389–429,
2011.

\bibitem{roelcke} W. Roelcke. \emph{\"Uber die Wellengleichung bei Grenzkreisgruppen erster Art}, S.-B. Heidelberger Akad. Wiss. Math.-Nat. Kl. 1953/1955 (1956), 159–267.



 


\bibitem{selberg}A. Selberg. 
 \emph{Equidistribution in discrete groups and the spectral theory of
automorphic forms.}
\url{http://publications.ias.edu/selberg/section/2491.}


\bibitem{tsuzukiletter} M. Tsuzuki. \emph{Letter to Kimball Martin regarding the paper `A relative trace
formula for a compact Riemann surface' .} Unpublished.

\bibitem{tsuzuki}M. Tsuzuki. \emph{Spectral means for period integrals of wave functions on real hyperbolic spaces.} Journal of Number Theory, 129:1387–2438, 2009.

\bibitem{dlmf}F. W. J. Olver, A. B. Olde Daalhuis, D. W. Lozier, B. I. Schneider, R. F. Boisvert, C. W. Clark, B. R. Miller, B. V. Saunders, H. S. Cohl, and M. A. McClain, eds. 
\emph{NIST Digital Library of Mathematical Functions.} http://dlmf.nist.gov/, Release 1.1.7 of 2022-10-15.  
\end{thebibliography}
\end{document}